\documentclass[11pt]{article}

\usepackage{xcolor}
\definecolor{capa}{RGB}{50, 90, 160}
\definecolor{delftblue}{RGB}{0,61,165}
\usepackage[colorlinks=true,
  linkcolor=delftblue,
  urlcolor=delftblue,
  citecolor=delftblue]{hyperref}
\newcommand{\Id}{{{\color{red}\mathsf{I}}}}

\usepackage{enumitem}

\usepackage{multirow}
\usepackage{amsmath}
\usepackage{amsfonts}
\usepackage{amsmath}
\usepackage{amssymb}  
\usepackage{graphicx}
\usepackage{caption}
\usepackage{mathrsfs}
\usepackage{upgreek}
\usepackage{amsthm}
\usepackage{subfig}
\usepackage{booktabs}
\usepackage{authblk}
\usepackage{cite}
\usepackage{bbm}
\usepackage{url}

\usepackage[]{algorithm2e}

\newtheorem{theorem}{Theorem}[section]
\newtheorem{lemma}[theorem]{Lemma}
\newtheorem{proposition}[theorem]{Proposition}
\newtheorem{corollary}[theorem]{Corollary}
\newtheorem{remark}[theorem]{Remark}
\newtheorem{definition}[theorem]{Definition}

\usepackage{mathtools}
\usepackage{bm}
\newcommand{\rr}{\bm{r}}
\newcommand{\nn}{\mathbf{n}}
\newcommand{\aaa}{\mathbf{a}}
\newcommand{\vv}{\mathbf{v}}
\newcommand{\kk}{\mathbf{k}}
\newcommand{\bphi}{\bm{\phi}}
\newcommand{\Spot}{\mathcal{S}}
\newcommand{\Dpot}{\mathcal{D}}
\newcommand{\Gpot}{\mathcal{G}}          %
\newcommand{\Vop}{\mathsf{V}}
\newcommand{\Kop}{\mathsf{K}}
\newcommand{\Ktop}{\widetilde{\mathsf{K}}}
\newcommand{\Wop}{\mathsf{W}}
\newcommand{\Cop}{\mathsf{C}}            %
\newcommand{\Lbf}{\mathbb{L}}           %
\newcommand{\Ebf}{\mathbb{E}}
\renewcommand{\Id}{\mathsf{Id}}
\newcommand{\Neta}{\mathsf{N}_\eta}
\newcommand{\Ue}{U_{\mathrm{e}}}
\newcommand{\Rvec}{\bm{R}}
\newcommand{\bgamma}{\bm{\gamma}}     %

\title{Photonic band structure by direct second-kind free-space boundary integral equations}

\author[1]{Carlos P\'erez-Arancibia\thanks{\href{mailto:c.a.perezarancibia@utwente.nl}{c.a.perezarancibia@utwente.nl}}}
\author[2]{Tao Yin\thanks{\href{mailto:yintao@lsec.cc.ac.cn}{yintao@lsec.cc.ac.cn}}}
\affil[1]{\small{Department of Applied Mathematics, University of Twente, Enschede, The Netherlands}}
\affil[2]{\small{State Key Laboratory of Mathematical Sciences and Institute of Computational Mathematics and Scientific/Engineering Computing, Academy of Mathematics and Systems Science, Chinese Academy of Sciences, Beijing 100190, China}}
\date{\today}

\begin{document}

\maketitle
\begin{abstract}
We present a direct boundary integral equation (BIE) formulation of the Bloch eigenvalue problem for
two-dimensional photonic crystals that uses only the free-space Helmholtz Green function.  The
unknowns are the Cauchy data of the Bloch mode on the inclusion boundary and on two cell walls.  On the inclusion boundary, the transmission conditions are imposed via M\"uller's formulation.  The wall equations, in turn, come from a stack formulation, in which Green's formula is applied on stacks of three cells, so that the wall self-interactions and the interactions between
arcs sharing endpoints cancel exactly.  The resulting system is of the second kind, Fredholm of index
zero, and free of corner singularities, and the walls may be curved to keep the cell boundary away
from the inclusions.  Its operators are assembled once per frequency and reused for all Bloch vectors.  We prove that every Bloch mode yields a nonzero null vector, and that a null vector
comes from a Bloch mode exactly when its cell representation vanishes outside the cell, which gives an
inexpensive test for spurious null vectors.  The band frequencies are computed with a high-order
Nystr\"om discretization and the search of Bruno, Santana and Trefethen~\cite{BST2026}, which locates them as
poles of adaptive Antoulas--Anderson (AAA) rational approximants to a randomly scalarized resolvent
sampled on the real frequency axis.
Numerical examples validate the method against exact and published band structures and include a
challenging triangular lattice of air pores in a dielectric, whose cell requires curved walls.
\end{abstract}



\section{Introduction}
\label{sec:intro} Photonic crystals are periodic dielectric structures whose band structure, the relation between frequency and Bloch wave vector, determines how light propagates through them~\cite{JJWM2008,Kuc2001}. In two dimensions, computing the band structure amounts to solving, for each Bloch vector in the Brillouin zone, a Helmholtz eigenvalue problem on a unit cell with quasi-periodic boundary conditions. Standard methods for this computation include plane-wave expansions, as implemented in MPB~\cite{JJ2001}, finite elements~\cite{AK1999,Dobson1999,DBB2006}, and finite-difference time-domain (FDTD) methods~\cite{TH2005,Oskooi2010}. All of them discretize the entire cell, so resolving high dielectric contrast, curved interfaces or small geometric features may require a large number of unknowns. Moreover, plane-wave expansions converge slowly across material discontinuities, finite-difference grids represent curved interfaces only to low order, and finite element and finite-difference methods suffer from numerical dispersion, which becomes the pollution effect at high frequencies~\cite{BS1997}. 

Boundary integral equations (BIEs) discretize only interfaces and achieve high-order convergence for smooth geometries, at the cost of nonlinear dependence on frequency. The natural kernel is the quasi-periodic Green function, whose slowly convergent lattice sums have received considerable attention~\cite{CNM1994,MNBG2000,DHH2001,Moroz2001}. However, this Green function does not exist at empty-lattice eigenvalues, the analogue of Wood anomalies~\cite{LT2007}, and formulations based on it become ill-conditioned at frequencies close to these configurations~\cite{BG2010}. Related approaches use Rayleigh multipole and multiple-scattering expansions~\cite{LQ1993,NMB1995,Botten2003,SHE2003}, cell Dirichlet-to-Neumann maps~\cite{YLA2008}, wavelet-based boundary elements~\cite{WMX2022}, and periodic fast multipole acceleration~\cite{ON2008}. Since the Bloch problem is an eigenvalue problem that depends nonlinearly on the frequency, it can be treated by nonlinear eigenvalue methods~\cite{SP2005,GT2017}, contour integration~\cite{Beyn2012,ASTIK2009}, or rational approximation~\cite{CT2025,BST2026}; BIE algorithms and analysis for such eigenvalue problems are developed in~\cite{ABN2015,SU2012,SU2017}. An indirect formulation using only the free-space Green function was developed in~\cite{BG2010}. It combines inclusion layer potentials summed over a phased $3\times3$ block of nearest images with auxiliary single- and double-layer densities on two cell walls and their extensions. Enforcing transmission and quasi-periodicity yields a second-kind system with smooth densities, spectral accuracy, and dependence on the Bloch multipliers through finitely many monomials. Bands are located by scanning the smallest singular value of the system.

This paper develops a direct BIE counterpart, also based only on the free-space Green function, by extending the free-space formulation of~\cite{SFFP} from line arrays of penetrable obstacles to doubly periodic lattices. The unknowns are the Cauchy data of the Bloch mode on the inclusion boundary and two cell walls. Combining Green representations on the cell and three-cell stacks yields a second-kind BIE system without corner singularities, whose operator coefficients are assembled once per frequency and reused across Bloch vectors. We discretize this system by a high-order Nystr\"om method. For each Bloch vector, the band frequencies are the real frequencies at which the discrete operator is singular, and we locate them, following~\cite{BST2026}, as poles of AAA rational approximants to a randomly scalarized resolvent sampled along the real frequency axis. The direct formulation guarantees that every Bloch mode yields a nonzero null vector (Lemma~\ref{lem:necessity}), so no band is missed; for indirect formulations, the corresponding statement, that every Bloch mode admits a density representation, requires a separate argument (cf.~\cite[Conjecture~7]{BG2010}). The converse remains open, but we prove that every null vector satisfies the transmission conditions and that spurious null vectors are characterized by nonzero radiation from the cell representation, which provides an inexpensive a posteriori test. On the computational side, the stack representations involve no interactions between the inclusion and its nearest images, and curved walls, combined with the density interpolation method for nearly singular integrals~\cite{FPB2021}, accommodate inclusions that do not fit inside the primitive parallelogram.

The paper is organized as follows. Section~\ref{sec:problem} states the Bloch eigenvalue problem. Section~\ref{sec:layer-rep} introduces the layer potentials and the translation notation, and derives the Green's representations on which the formulation rests. Section~\ref{sec:bie} assembles the BIE system and establishes its symmetries and Fredholm property, and \S\ref{sec:spurious} characterizes its spurious null vectors. Section~\ref{sec:numerics} describes the discretization and the eigenvalue search, \S\ref{sec:experiments} presents the numerical experiments, and \S\ref{sec:conclusion} concludes.

\section{The Bloch eigenvalue problem} \label{sec:problem}
Let $\aaa_1,\aaa_2\in\mathbb{R}^2\setminus\{\mathbf 0\}$ be linearly independent primitive lattice vectors,
\begin{equation}
  \Lambda := \{\, \vv_{mn} := m\aaa_1+n\aaa_2 \;:\; (m,n)\in\mathbb{Z}^2 \,\},
  \label{eq:lattice}
\end{equation}
with reciprocal lattice $\Lambda^*=\{\mathbf{G}=m\mathbf{b}_1+n\mathbf{b}_2\}$ defined by $\aaa_i\cdot\mathbf{b}_j=2\pi\delta_{ij}$, where $\delta_{ij}$ is the Kronecker delta.
Let $U$ be a primitive cell of $\Lambda$: a bounded domain whose boundary is a simple closed curve and
whose lattice translates $\overline U+\vv$, $\vv\in\Lambda$, cover the plane with pairwise disjoint
interiors.  Let
$\Omega\subset U$ be the dielectric inclusion, with smooth boundary and
$\overline{\Omega}\cap\partial U=\emptyset$, and
\begin{equation}
  \Gamma_0 := \partial\Omega,
  \qquad
  D_2 := \bigcup_{\vv\in\Lambda}(\Omega+\vv),
  \qquad
  D_1 := \mathbb{R}^2\setminus\overline{D_2}.
\end{equation}

We require the cell walls to come in translation pairs, one per lattice direction:
\begin{equation}
  \partial U=\Gamma_1\cup(\Gamma_1+\aaa_1)\cup\Gamma_2\cup(\Gamma_2+\aaa_2),
  \label{eq:walls}
\end{equation}
so that the wall $\Gamma_j$ is paired with the lattice vector $\aaa_j$, $j\in\{1,2\}$, and the inclusion boundary carries the index $0$.  We assume that the four arcs in~\eqref{eq:walls} are distinct and meet only at their endpoints. Together with~$\Gamma_0$, the geometry is described by the three curves parametrized by
\begin{equation}
  \mathbf{r}_i=(\mathsf{x}_i,\mathsf{y}_i):I_i\to\Gamma_i,\quad i\in\{0,1,2\},
  \label{eq:wall-param}
\end{equation}
all regular and, for simplicity, $C^\infty$ (the walls up to their endpoints), with $I_0:=[0,2\pi)$ and $\mathbf{r}_0$ $2\pi$-periodic and positively oriented. The translated walls
$\Gamma_1+\aaa_1$ and $\Gamma_2+\aaa_2$ are therefore parameterized by $\mathbf{r}_1+\aaa_1$ on $I_1$ and $\mathbf{r}_2+\aaa_2$ on $I_2$.  We use the parametric normals
\begin{equation}
  \nn_i:=(\mathsf{y}'_i,-\mathsf{x}'_i)/|\mathbf{r}'_i|,\qquad i\in\{0,1,2\},
  \label{eq:normals-match}
\end{equation}
which a translated wall inherits unchanged, together with the speed $|\mathbf{r}'_i|$. Writing $\mathbf{c}$ for the junction of $\Gamma_1$ and $\Gamma_2$, the four corners of the cell are  $\mathbf{c}$, $\mathbf{c}+\aaa_1$,  $\mathbf{c}+\aaa_2$,  and $\mathbf{c}+\aaa_1+\aaa_2$. The geometry is sketched in Figure~\ref{fig:geometry}.

\begin{figure}[tb]
\centering
\includegraphics{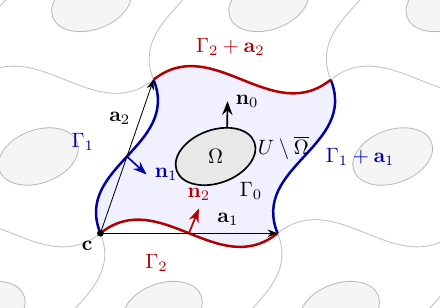}
\caption{Primitive cell $U$ (shaded) and parts of the neighbouring cells.  The lattice vectors $\aaa_1,\aaa_2$
start at the corner $\mathbf{c}$, and the inclusion $\Omega$ has boundary $\Gamma_0$ with outward normal
$\nn_0$.  The walls $\Gamma_1,\Gamma_1+\aaa_1$ (blue) and $\Gamma_2,\Gamma_2+\aaa_2$ (red) form the
translation pairs of \eqref{eq:walls}, with normals $\nn_1$ and $\nn_2$ pointing in the directions of
$\aaa_1$ and $\aaa_2$.  The walls may be curved, so that $\partial U$ does not intersect
$\overline{\Omega}$.}
\label{fig:geometry}
\end{figure}

Using the time convention $\mathrm{e}^{-i\omega t}$ and denoting the wavenumbers by
\begin{equation}
  k_1=\omega\sqrt{\varepsilon_1\mu_1} \;\text{ in } D_1,
  \qquad
  k_2=\omega\sqrt{\varepsilon_2\mu_2} \;\text{ in } D_2,
  \qquad \varepsilon_1,\varepsilon_2,\mu_1,\mu_2>0,
  \label{eq:wavenumbers}
\end{equation}
the field of interest $u:\mathbb{R}^2\to\mathbb{C}$---the out-of-plane component $E_z$ of the
electric field in E-polarization, $H_z$ of the magnetic field in H-polarization---is classical on each side
of the interface, i.e.,
\begin{equation}
  u|_{D_1}\in C^2(D_1)\cap C^1(\overline{D_1})
  \qquad\text{and}\qquad
  u|_{D_2}\in C^2(D_2)\cap C^1(\overline{D_2}).
  \label{eq:regularity}
\end{equation}
For positive physical frequencies $\omega>0$, both $k_1$ and $k_2$ are positive.
The physical band structure includes $\omega=0$, but the Hankel-kernel formulation
below is used only for $\omega>0$ and its analytic continuation to
$\mathbb C\setminus(-\infty,0]$ (\S\ref{sec:fredholm}).  We do not formulate the
zero-frequency limit of these layer potentials here.

In the sequel we let $\nn$ denote the unit normal on $\partial D_2$ directed out of $D_2$. It is
$\Lambda$-periodic by construction and, on $\Gamma_0$, agrees with the parametric normal
$\nn_0$ of \eqref{eq:normals-match}.  

For a curve $\Gamma$ carrying a chosen unit normal
$\nn$ we define the Dirichlet and Neumann traces of $u$ on $\Gamma$ by
\begin{equation}
  \big(\gamma^\pm_{D,\Gamma}u\big)(\rr):=\lim_{\rho\to0^+}u\big(\rr\pm\rho\nn(\rr)\big),
  \quad
  \big(\gamma^\pm_{N,\Gamma}u\big)(\rr):=\lim_{\rho\to0^+}\nabla u\big(\rr\pm\rho\nn(\rr)\big)\cdot\nn(\rr),
  \quad \rr\in\Gamma.
  \label{eq:traces}
\end{equation}
The limits exist by \eqref{eq:regularity}, and on $\partial D_2$ the superscript ``$+$'' selects the limit taken from $D_1$ and ``$-$'' the limit taken from $D_2$.
The same definition is used on the open wall arcs $\Gamma_\ell+\vv$, $\ell\in\{1,2\}$, with
$\nn=\nn_\ell$ and ``$+$'' the side into which $\nn_\ell$ points. 

For a field that is smooth across the wall, such as a Bloch mode, the two limits agree.
On the curves $\Gamma_\ell$, $\ell=0,1,2$, we always take $\nn=\nn_\ell$ and use the
parametric traces
\begin{equation}
  \gamma^\pm_{D,\ell}u:=\big(\gamma^\pm_{D,\Gamma_\ell}u\big)\circ\mathbf{r}_\ell,
  \qquad
  \gamma^\pm_{N,\ell}u:=\big(\gamma^\pm_{N,\Gamma_\ell}u\big)\circ\mathbf{r}_\ell,
  \label{eq:param-traces}
\end{equation}
which are functions on $I_\ell$. Here the superscript $\pm$ is dropped whenever the two one-sided traces agree.  On a translated curve $\Gamma_\ell+\vv$ we write likewise $\gamma^\pm_{D,\ell}[\vv]u$ and $\gamma^\pm_{N,\ell}[\vv]u$ for the traces taken with the normal $\nn_\ell$ and composed with $\mathbf{r}_\ell+\vv$.  The two traces are collected into the Cauchy-data pair $\bgamma^\pm_\ell u:=(\gamma^\pm_{D,\ell}u,\gamma^\pm_{N,\ell}u)$, and likewise $\bgamma^\pm_\ell[\vv]u$.

With this notation the sought field $u:\mathbb R^2\to\mathbb C$ satisfies
\begin{subequations}\label{eq:bvp}
\begin{align}
  \Delta u+k_1^2u=0 \;\text{ in }D_1,
  &\qquad
  \Delta u+k_2^2u=0 \;\text{ in }D_2,
  \label{eq:bvp-helmholtz}\\
  \gamma^+_{D,\partial D_2}u=\gamma^-_{D,\partial D_2}u,
  &\qquad
  \gamma^+_{N,\partial D_2}u=\eta\,\gamma^-_{N,\partial D_2}u,
  \label{eq:bvp-transmission}\\
  u(\rr+\aaa_1)=\zeta_1 u(\rr),&\qquad u(\rr+\aaa_2)=\zeta_2 u(\rr),
  \qquad \rr\in\mathbb{R}^2,
  \label{eq:bvp-bloch}
\end{align}
\end{subequations}
with $\eta:=\mu_1/\mu_2$ (E-polarization) or $\varepsilon_1/\varepsilon_2$ (H-polarization).  The two Bloch phases in~\eqref{eq:bvp-bloch} are
\begin{equation}
\zeta_1:=\mathrm{e}^{i\kk\cdot\aaa_1}\qquad\text{and}\qquad \zeta_2:=\mathrm{e}^{i\kk\cdot\aaa_2},
  \label{eq:zeta-def}
\end{equation}
and we write $\zeta:=(\zeta_1,\zeta_2)$ for the pair.  Since $\zeta_1^m\zeta_2^n=\mathrm{e}^{i\kk\cdot\vv_{mn}}$,
\eqref{eq:bvp-bloch} reads $u(\rr+\vv)=\mathrm{e}^{i\kk\cdot\vv}u(\rr)$ for $\vv\in\Lambda$.  
Note that, by \eqref{eq:bvp-bloch}, it suffices to impose \eqref{eq:bvp-transmission} on the single
curve $\Gamma_0$, and differentiating \eqref{eq:bvp-bloch} gives
the same relations for $\nabla u$, so they need not be imposed separately.

Problem \eqref{eq:bvp} is the standard Bloch eigenvalue problem for a two-dimensional photonic crystal~\cite{Kuc2001,JJWM2008}.  Note that~\eqref{eq:bvp} is homogeneous, so $u\equiv0$ always solves it. What is sought is then not the field for prescribed parameters, but the parameters admitting a nontrivial one. In detail:
\begin{definition}[Bloch eigenvalue problem]
\label{def:bloch}
The lattice $\Lambda$, the inclusion $\Omega$ and the material constants
$\varepsilon_s,\mu_s$ being fixed, call a pair
$(\omega,\zeta)\in[0,\infty)\times\mathbb{T}^2$, $\mathbb{T}:=\{z\in\mathbb{C}:|z|=1\}$,
a \emph{Bloch eigenvalue},
and any corresponding $u\ne0$ a \emph{Bloch mode}, if \eqref{eq:bvp} admits a solution
$u\ne0$ of the regularity \eqref{eq:regularity}.  The set of all such pairs,
\begin{equation}
  \mathcal{B}:=\big\{(\omega,\zeta)\in[0,\infty)\times\mathbb{T}^2\;:\;\eqref{eq:bvp}\ \text{has a solution } u\ne0\big\},
  \label{eq:blochvariety}
\end{equation}
is the \emph{physical dispersion relation} of the crystal, and the Bloch modes at
a given $(\omega,\zeta)\in\mathcal{B}$ span its eigenspace there.
\end{definition}

Since \eqref{eq:bvp} depends on $\omega$ only through $\omega^2$, we take $\omega\ge0$.  For $\zeta\in\mathbb T^2$ and positive material constants, the energy identity over the cell (a Rayleigh quotient) shows that $\omega^2$ is real and nonnegative~\cite{Kuc2001,JJWM2008}, so no Bloch eigenvalue lies off the real axis. 

By \eqref{eq:zeta-def} the constraint $\zeta\in\mathbb{T}^2$ is exactly the statement that $\kk$ is real, and since~$\zeta$ depends on $\kk$ only modulo $\Lambda^*$ it
suffices to let $\kk$ range over the first Brillouin zone~\cite{JJWM2008}.  For $\varepsilon_s,\mu_s>0$ the admissible frequencies at a given $\kk$ then form a discrete sequence
\begin{equation}
  0\le\omega_1(\kk)\le\omega_2(\kk)\le\cdots\longrightarrow\infty ,
  \label{eq:bands}
\end{equation}
the \emph{bands}, and $\mathcal{B}$ is their graph over the Brillouin zone, i.e., the band structure.  The discreteness of the spectrum at fixed $\kk$ is classical, and we refer to \cite{Kuc1993} for the general theory and to \cite{Kuc2001} for the photonic-crystal case, where the band structure is identified with the spectrum of that family.  

What follows is aimed at recasting~\eqref{eq:blochvariety}, which is a statement about a~\eqref{eq:bvp}, as a condition that a certain (block) boundary integral operator have nontrivial kernel.

\begin{remark}
We assume smooth curves and classical solutions for simplicity, and because the high-order quadratures of \S\ref{sec:discretization} require smooth boundaries.  The layer potentials and boundary integral operators used below extend to Lipschitz curves in the standard Sobolev-space setting~\cite{Costabel1988,McLean2000}, but the compactness and regularity arguments of \S\ref{sec:fredholm} rely on smoothness %
and different arguments would be needed.
\end{remark}

\section{Layer potentials and Green's representations}
\label{sec:layer-rep}

\subsection{Layer potentials, integral operators, and translation}
\label{sec:operators}
Throughout, $s\in\{1,2\}$ indexes the wavenumber $k_s$ of \eqref{eq:wavenumbers}, and $\ell,i\in\{0,1,2\}$ index the curves $\Gamma_0,\Gamma_1,\Gamma_2$ of \S\ref{sec:problem}, with $\ell$ labelling the target curve and $i$ the source. Since quasi-periodicity is imposed through the parametrizations \eqref{eq:wall-param}, we write all densities and operators in parametric form. A density on $\Gamma_i$ is a function $\phi:I_i\to\mathbb{C}$, and a trace on $\Gamma_\ell$ is composed with $\mathbf{r}_\ell$, so that it is again a function on $I_\ell$.  

The Helmholtz single- and double-layer potentials over $\Gamma_i$ at wavenumber $k_s$~\cite{CK1983} are given by
\begin{subequations}\label{eq:potentials}
\begin{align}
  (\Spot^i_s\phi)(\rr)
    &= \frac{i}{4}\int_{I_i}H^{(1)}_0\big(k_s|\rr-\mathbf{r}_i(\tau)|\big)\,
       \phi(\tau)\,|\mathbf{r}'_i(\tau)|\,\mathrm{d}\tau\quad\text{and}\\
  (\Dpot^i_s\phi)(\rr)
    &= \frac{ik_s}{4}\int_{I_i}H^{(1)}_1\big(k_s|\rr-\mathbf{r}_i(\tau)|\big)\,
       \frac{\big(\rr-\mathbf{r}_i(\tau)\big)\cdot\nn_i(\tau)}{|\rr-\mathbf{r}_i(\tau)|}\,
       \phi(\tau)\,|\mathbf{r}'_i(\tau)|\,\mathrm{d}\tau,
\end{align}
\end{subequations}
for $\rr\in\mathbb{R}^2\setminus\Gamma$, and the four boundary integral operators with
target $\Gamma_\ell$ and source $\Gamma_i$ are
\begin{equation}
\begin{aligned}
  (\mathsf{A}^{\ell,i}_s\phi)(t) &= \int_{I_i}Q^{\ell,i}_{A,s}(t,\tau)\,\phi(\tau)\,|\mathbf{r}'_i(\tau)|\,\mathrm{d}\tau,
  \qquad \mathsf{A}\in\{\Vop,\Kop,\Ktop\},\\
  (\Wop^{\ell,i}_s\phi)(t) &= \mathrm{p.f.}\!\int_{I_i}Q^{\ell,i}_{W,s}(t,\tau)\,\phi(\tau)\,|\mathbf{r}'_i(\tau)|\,\mathrm{d}\tau,
\end{aligned}
  \label{eq:param-operators}
\end{equation}
for $t\in I_\ell$. These are, respectively, the single-layer, double-layer, adjoint double-layer and
hypersingular operators, $Q^{\ell,i}_{A,s}$ being the kernel belonging to $\mathsf{A}$ and ``$\mathrm{p.f.}$'' denoting that the integral is to be understood in the sense of the Hadamard finite part.

The classical jump relations \cite{CK1983} become, in this notation,
\begin{equation}
\begin{aligned}
  \gamma^\pm_{D,\ell}\Spot^i_s\phi &= \Vop^{\ell,i}_s\phi,
  &\quad
  \gamma^\pm_{N,\ell}\Spot^i_s\phi &= \mp\delta_{i,\ell}\frac{\phi}{2}+\Ktop^{\ell,i}_s\phi,\\
  \gamma^\pm_{N,\ell}\Dpot^i_s\phi &= \Wop^{\ell,i}_s\phi,
  &\quad
  \gamma^\pm_{D,\ell}\Dpot^i_s\phi &= \pm\delta_{i,\ell}\frac{\phi}{2}+\Kop^{\ell,i}_s\phi.
\end{aligned}
  \label{eq:jumps}
\end{equation}
On the open wall arcs the relations hold at interior points of the arc, for densities smooth up to its endpoints.  At the endpoints themselves, however, the potentials may be singular.  Writing
$\Rvec_{\ell,i}:=\mathbf{r}_\ell(t)-\mathbf{r}_i(\tau)$ and $R_{\ell,i}:=|\Rvec_{\ell,i}|$,
the kernels of the operators in  \eqref{eq:param-operators} are 
\begin{subequations}\label{eq:kernels}
\begin{align}
  Q^{\ell,i}_{V,s}(t,\tau) &:= \frac{i}{4}H^{(1)}_0(k_sR_{\ell,i}),
  \label{eq:kernel-V}\\
  Q^{\ell,i}_{K,s}(t,\tau) &:= \frac{ik_s}{4}H^{(1)}_1(k_sR_{\ell,i})\,
    \frac{\Rvec_{\ell,i}\cdot\nn_i(\tau)}{R_{\ell,i}},
  \label{eq:kernel-K}\\
  Q^{\ell,i}_{\widetilde{K},s}(t,\tau) &:= -\frac{ik_s}{4}H^{(1)}_1(k_sR_{\ell,i})\,
    \frac{\Rvec_{\ell,i}\cdot\nn_\ell(t)}{R_{\ell,i}}, \quad \text{and}
  \label{eq:kernel-Kt}\\
  Q^{\ell,i}_{W,s}(t,\tau) &:= \frac{ik_s}{4}\Bigg(
    \frac{H^{(1)}_1(k_sR_{\ell,i})}{R_{\ell,i}}\,\nn_\ell(t)\cdot\nn_i(\tau)
    \nonumber\\
  &\hspace{2.0em}
    +\Big\{k_sR_{\ell,i}H^{(1)}_0(k_sR_{\ell,i})-2H^{(1)}_1(k_sR_{\ell,i})\Big\}
    \nonumber\\
  &\hspace{6.0em}
    \times\frac{\big(\Rvec_{\ell,i}\cdot\nn_i(\tau)\big)\big(\Rvec_{\ell,i}\cdot\nn_\ell(t)\big)}{R_{\ell,i}^{3}}
    \Bigg).
  \label{eq:kernel-W}
\end{align}
\end{subequations}
All four kernels are $C^\infty$ up to the parameter endpoints when the closures of the target and source curves are disjoint.  Distinct wall arcs
can have singular interactions at shared endpoints; these interactions are removed by the stack construction below.  In the self-interaction case $\ell=i$, $Q^{\ell,i}_{V,s}$ is logarithmically singular, $Q^{\ell,i}_{K,s}$ and $Q^{\ell,i}_{\widetilde{K},s}$ extend continuously, and only
$Q^{\ell,i}_{W,s}$ is hypersingular; in our formulation this singularity cancels, as in the M\"uller formulation of~\cite{KR1978}.%

The field representations of \S\ref{sec:representation} integrate over lattice translates of the
curves of \S\ref{sec:problem}.  Rather than labelling the translated curves, we track the translation
through an additional argument: for $\vv\in\mathbb{R}^2$, $\Spot^i_s[\vv]$, $\Dpot^i_s[\vv]$ and
$\mathsf{A}^{\ell,i}_s[\vv]$, $\mathsf{A}\in\{\Vop,\Kop,\Ktop,\Wop\}$, denote the potentials
\eqref{eq:potentials} and operators \eqref{eq:param-operators} with source curve $\Gamma_i+\vv$,
parametrized by $\mathbf{r}_i+\vv$.  Superscripts thus always refer to the untranslated curves, and
$[\mathbf{0}]$ is omitted.  Since translation changes neither $\nn_i$ nor $|\mathbf{r}'_i|$, the kernel
of $\mathsf{A}^{\ell,i}_s[\vv]$ is that of \eqref{eq:kernels} with $\Rvec_{\ell,i}$ replaced by
$\Rvec_{\ell,i}-\vv$, and
\begin{equation}
  (\Spot^i_s[\vv]\phi)(\rr)=(\Spot^i_s\phi)(\rr-\vv),
  \qquad
  (\Dpot^i_s[\vv]\phi)(\rr)=(\Dpot^i_s\phi)(\rr-\vv).
  \label{eq:shift}
\end{equation}

\paragraph{Conventions} To keep the notation manageable, we adopt two conventions.  First, every
operator in the formulation is taken at the exterior wavenumber $k_1$, except the self-interaction
of $\Gamma_0$ at $k_2$, so we drop the subscript $s=1$,
\begin{equation*}
  \Spot^i[\vv]:=\Spot^i_1[\vv],\qquad \Dpot^i[\vv]:=\Dpot^i_1[\vv],\qquad
  \mathsf{A}^{\ell,i}[\vv]:=\mathsf{A}^{\ell,i}_1[\vv],
\end{equation*}
and keep the subscript $2$ where it occurs.  Second, all potentials act on Cauchy data: for a pair
$\Phi=(\phi,\psi)$ of densities on $I_i$ we define the Green layer potential and its
$2\times2$ trace operator
\begin{equation}
  \Gpot^i[\vv]\Phi:=\Dpot^i[\vv]\phi-\Spot^i[\vv]\psi,
  \qquad
  \Cop^{\ell,i}[\vv]:=
  \begin{bmatrix}
    \Kop^{\ell,i}[\vv] & -\Vop^{\ell,i}[\vv]\\[2pt]
    \Wop^{\ell,i}[\vv] & -\Ktop^{\ell,i}[\vv]
  \end{bmatrix},
  \label{eq:green-layer}
\end{equation}
and likewise $\Gpot^0_2$ and $\Cop^{0,0}_2$ at $k_2$.  Green's representation formula is then a sum
of terms $\Gpot^i[\vv]\Phi$, and \eqref{eq:jumps} becomes
\begin{equation}
  \bgamma^\pm_\ell\,\Gpot^i\Phi=\pm\tfrac12\,\delta_{i,\ell}\,\Phi+\Cop^{\ell,i}\Phi .
  \label{eq:jumps-pair}
\end{equation}

\begin{lemma}%
\label{lem:translation}
For all $\ell,i\in\{0,1,2\}$, $\vv,\vv'\in\mathbb{R}^2$ and $\mathsf{A}\in\{\Vop,\Kop,\Ktop,\Wop\}$,
the operator with target $\Gamma_\ell+\vv'$ (parametrized by $\mathbf{r}_\ell+\vv'$) and source
$\Gamma_i+\vv$ is $\mathsf{A}^{\ell,i}[\vv-\vv']$, wherever the latter is defined.  For lattice translations
$\vv,\vv'\in\Lambda$, at interior points of the target arc,
\begin{equation}
  \bgamma^\pm_\ell[\vv']\,\Gpot^i[\vv]\Phi
  =\pm\tfrac12\,\delta_{i,\ell}\,\delta_{\vv,\vv'}\,\Phi+\Cop^{\ell,i}[\vv-\vv']\Phi .
  \label{eq:translation}
\end{equation}
\end{lemma}

\begin{proof}
Translating the target by $\vv'$ and the source by $\vv$ replaces $\mathbf{r}_\ell(t)-\mathbf{r}_i(\tau)$ by $\mathbf{r}_\ell(t)-\mathbf{r}_i(\tau)-(\vv-\vv')$ and
leaves normals and speeds unchanged.  For lattice translates in the assumed tiling, distinct source and target arcs have disjoint interiors and they coincide only when $i=\ell$ and $\vv=\vv'$.  Thus \eqref{eq:translation} follows from~\eqref{eq:jumps-pair}. 
\end{proof}

\subsection{Three Green's representations}
\label{sec:representation}

Let $u$ be a Bloch mode, i.e., a nontrivial solution of \eqref{eq:bvp}.  We define its Cauchy data on
the three curves by
\begin{equation}
  \Phi_\ell=(\phi_\ell,\psi_\ell):=\bgamma^-_\ell u\quad\text{on } I_\ell,\qquad \ell\in\{0,1,2\}.
  \label{eq:densities}
\end{equation}

By \eqref{eq:bvp-transmission}
the exterior data on $\Gamma_0$ are $\Neta\Phi_0$, with $\Neta:=\operatorname{diag}(1,\eta)$, and by \eqref{eq:bvp-bloch}
\begin{equation}
  \bgamma^\pm_i[\vv]u=\mathrm{e}^{i\kk\cdot\vv}\,\bgamma^\pm_iu,\qquad \vv\in\Lambda,
  \label{eq:extension}
\end{equation}
which is the density-transfer rule of \cite[\S8.1]{SFFP}.

Let $\mathcal{P}$ be the interior of $\bigcup_{\vv\in F}(\overline{U}+\vv)$ for a finite set
$F\subset\Lambda$, and $\mathcal{P}^\circ:=\mathcal{P}\setminus\overline{D_2}$ the same region with the inclusions removed.  Green's representation formula~\cite{CK1983} in $\mathcal{P}^\circ$ with the free-space kernel,
\begin{equation}
  G_k(\rr,\rr'):=\tfrac{i}{4}H^{(1)}_0(k|\rr-\rr'|),\qquad \rr\ne\rr',
  \label{eq:greenfn}
\end{equation}
yields, with $\bm{\nu}$ the unit normal directed out of $\mathcal{P}^\circ$,
\begin{equation}\label{eq:Green_formula}
  \int_{\partial\mathcal{P}^\circ}\Big\{G_{k_1}(\rr,\rr')\,\frac{\partial u}{\partial\bm{\nu}}(\rr')
   -u(\rr')\,\frac{\partial G_{k_1}(\rr,\rr')}{\partial\bm{\nu}(\rr')}\Big\}\,\mathrm{d}s(\rr')
  =\begin{cases} u(\rr), & \rr\in\mathcal{P}^\circ,\\ 0, & \rr\notin\overline{\mathcal{P}^\circ}.\end{cases}
\end{equation}
The boundary $\partial\mathcal{P}^\circ$ consists of the inclusion boundaries $\Gamma_0+\vv$,
$\vv\in F$, and of the wall arcs $\Gamma_1+\vv$ and $\Gamma_2+\vv$ that form $\partial\mathcal{P}$.  On each
such arc $\Gamma_i+\vv$ we express the integrand of \eqref{eq:Green_formula} through the fixed normal
$\nn_i$ of \S\ref{sec:problem} instead of $\bm{\nu}$.  Since $\bm{\nu}=-\sigma\nn_i$ with
\begin{equation}
  \sigma:=\begin{cases}
    +1, & \nn_i \text{ points into } \mathcal{P}^\circ,\\
    -1, & \nn_i \text{ points out of } \mathcal{P}^\circ,
  \end{cases}
  \label{eq:orientation-sign}
\end{equation}
the arc contributes $\sigma\,\Gpot^i[\vv]\Psi$ to the left-hand side of \eqref{eq:Green_formula}, where
$\Psi$ is the Cauchy data of $u$ on the arc, taken from $\mathcal{P}^\circ$ with respect to $\nn_i$.  By \eqref{eq:densities} and \eqref{eq:extension},
$\Psi=\mathrm{e}^{i\kk\cdot\vv}\Phi_i$ on a wall arc and $\Psi=\mathrm{e}^{i\kk\cdot\vv}\Neta\Phi_0$ on an inclusion boundary.  

Since $\nn_0$ points out of $\Omega$, and $\nn_j$ in the sense of $+\aaa_j$, $j\in\{1,2\}$ (Figure~\ref{fig:geometry}), every inclusion boundary has $\sigma=+1$, and a wall arc $\Gamma_j+\vv$ has $\sigma=+1$ if $\mathcal{P}$ lies on its $+\aaa_j$ side and $\sigma=-1$ otherwise.  

The left-hand side of~\eqref{eq:Green_formula} therefore depends on $u$ only through the density
vector $\bphi:=(\Phi_0,\Phi_1,\Phi_2)$.  For any $\bphi$ in
\begin{equation}
  \mathbf{X}:=\prod_{\ell=0}^{2}\Big(C^{\infty}(\overline{I_\ell})\times C^{\infty}(\overline{I_\ell})\Big),
  \label{eq:Xspace}
\end{equation}
with the components on $I_0$ $2\pi$-periodic, we define the representation on $\mathcal P$,
\begin{equation}
  u_{\mathcal{P}}[\bphi]:=\sum_{\vv\in F}\mathrm{e}^{i\kk\cdot\vv}\,\Gpot^0[\vv]\Neta\Phi_0
  +\sum_{j=1}^{2}\;\sum_{\Gamma_j+\vv\subset\partial\mathcal{P}}\sigma\,\mathrm{e}^{i\kk\cdot\vv}\,\Gpot^j[\vv]\Phi_j,
  \label{eq:rep-P}
\end{equation}
with $\Gpot^i[\vv]$ the Green layer potentials \eqref{eq:green-layer} and $\sigma$ as in \eqref{eq:orientation-sign}.  It is defined off the arcs, and its one-sided traces on
the open arcs exist and satisfy \eqref{eq:jumps}~\cite[\S2]{CK1983}.  For a Bloch mode with Cauchy data
$\bphi$, $u_{\mathcal P}[\bphi]$ is the left-hand side of \eqref{eq:Green_formula}.

In what follows, $j$ indexes a lattice direction, and with it the translation pair of walls
$\Gamma_j,\Gamma_j+\aaa_j$, and $j'$ indexes the other direction,
\begin{equation}
  j\in\{1,2\},\qquad j':=3-j.
  \label{eq:jprime}
\end{equation}

Three different domains $\mathcal P$ are used in our direct BIE formulation: the unit cell $U$ itself and, for each wall
$\Gamma_j$, the stack of three cells along it,
\begin{equation}
  S_j:=\operatorname{int}\bigcup_{n=-1}^{1}\big(\overline{U}+n\aaa_{j'}\big),
  \label{eq:stacks}
\end{equation}
so that $S_1$ is a ``column" and $S_2$ a ``row" (Figure~\ref{fig:stacks}).

\begin{figure}[tb]
\centering
\includegraphics{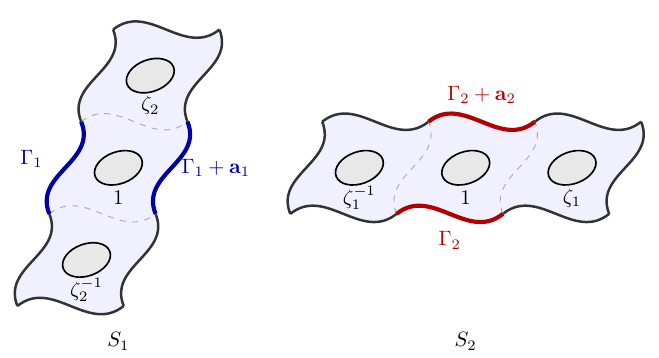}
\caption{The column $S_1$ and the row $S_2$ of \eqref{eq:stacks}, with the walls and inclusion of
Figure~\ref{fig:geometry} and the Bloch phase of each cell.  Black curves are the sources of the stack
representations \eqref{eq:rep-stack}, and blue and red curves are the targets of the wall equations for
$\Gamma_1$ and $\Gamma_2$.  Dashed walls lie inside the stack and are not sources.}
\label{fig:stacks}
\end{figure}

For $\mathcal{P}=U$, the wall arcs are $\Gamma_j$ ($\sigma=+1$) and $\Gamma_j+\aaa_j$ ($\sigma=-1$), and
\eqref{eq:rep-P} becomes the \emph{cell representation}
\begin{equation}
  u_U[\bphi]=\Gpot^0\Neta\Phi_0
  +\sum_{j=1}^{2}\big(\Gpot^j-\zeta_j\Gpot^j[\aaa_j]\big)\Phi_j ,
  \label{eq:rep-cell}
\end{equation}
and Green's representation formula in $\Omega$, i.e., \eqref{eq:Green_formula} with $\mathcal P^\circ$ replaced by $\Omega$
and $k_1$ by $k_2$, gives the \emph{interior representation}
\begin{equation}
  u_\Omega[\bphi]:=-\Gpot^0_2\Phi_0 .
  \label{eq:rep-interior}
\end{equation}
The pair \eqref{eq:rep-cell} in $U\setminus\overline{\Omega}$ and \eqref{eq:rep-interior} in
$\Omega$, extended by \eqref{eq:bvp-bloch}, is the representation we use to evaluate the Bloch
mode (see Lemma~\ref{lem:green-extinction}).

Next, the boundary of $S_j$ minus its inclusions consists of the arcs $\Gamma_j+n\aaa_{j'}$
($\sigma=+1$) and $\Gamma_j+\aaa_j+n\aaa_{j'}$ ($\sigma=-1$), $n=-1,0,1$, which continue the wall and its translate, the two arcs $\Gamma_{j'}-\aaa_{j'}$ ($\sigma=+1$) and $\Gamma_{j'}+2\aaa_{j'}$ ($\sigma=-1$) that close the stack, and the inclusion boundaries $\Gamma_0+n\aaa_{j'}$.  For $\mathcal{P}=S_j$, \eqref{eq:rep-P} therefore becomes the \emph{stack representation}
\begin{multline}
  u_{S_j}[\bphi]=\sum_{n=-1}^{1}\zeta_{j'}^{n}\Big(\Gpot^0[n\aaa_{j'}]\Neta\Phi_0
    +\big(\Gpot^j[n\aaa_{j'}]-\zeta_j\Gpot^j[\aaa_j+n\aaa_{j'}]\big)\Phi_j\Big)\\
  +\big(\zeta_{j'}^{-1}\Gpot^{j'}[-\aaa_{j'}]-\zeta_{j'}^{2}\Gpot^{j'}[2\aaa_{j'}]\big)\Phi_{j'} .
  \label{eq:rep-stack}
\end{multline}

The following lemma summarizes what the three representations give for a Bloch mode.
\begin{lemma}%
\label{lem:green-extinction}
Let $u$ be a Bloch mode, i.e., a nontrivial solution of \eqref{eq:bvp}, and $\bphi$ its Cauchy data
\eqref{eq:densities}.  Then, for $\mathcal P\in\{U,S_1,S_2\}$, the cell and stack representations
\eqref{eq:rep-cell} and \eqref{eq:rep-stack} satisfy
\begin{equation}
  u_{\mathcal P}[\bphi]=\begin{cases} u & \text{in } \mathcal P\setminus\overline{D_2},\\
  0 & \text{in } (\mathcal P\cap D_2)\cup(\mathbb R^2\setminus\overline{\mathcal P}),\end{cases}
  \label{eq:extinction-cell}
\end{equation}
and the interior representation \eqref{eq:rep-interior} satisfies
\begin{equation}
  u_\Omega[\bphi]=\begin{cases} u & \text{in } \Omega,\\ 0 & \text{in } \mathbb R^2\setminus\overline\Omega.\end{cases}
  \label{eq:extinction-interior}
\end{equation}
\end{lemma}

\begin{proof}
Both are Green's formula \eqref{eq:Green_formula}, in $\mathcal P\setminus\overline{D_2}$ at wavenumber $k_1$ and
in $\Omega$ at $k_2$, which applies by \eqref{eq:bvp} and \eqref{eq:regularity}.
\end{proof}

\section{Boundary integral formulation} \label{sec:bie}
Having established above three field representations, we now derive a BIE system by suitably combining their traces.

\subsection{Transmission rows}
Setting the exterior trace of \eqref{eq:rep-cell} equal to $\Neta\Phi_0$, and the interior trace of~\eqref{eq:rep-interior} equal to $\Phi_0$, we obtain, via the jump ralations~\eqref{eq:jumps-pair}, the identities
\begin{equation}
  \tfrac12\Neta\Phi_0=\Cop^{0,0}\Neta\Phi_0+\sum_{j=1}^{2}\Lbf^{0,j}\Phi_j
  \quad\text{and}\quad
  \tfrac12\Phi_0=-\Cop^{0,0}_2\Phi_0,
  \label{eq:traces-transmission}
\end{equation}
with the blocks
\begin{equation}
  \Lbf^{0,j}:=\Cop^{0,j}-\zeta_j\Cop^{0,j}[\aaa_j],\qquad j=1,2.
  \label{eq:blocks-transmission}
\end{equation}
Adding the two identities gives the M\"uller equations~\cite{KR1978}, which in our notations take the form
\begin{subequations}\label{eq:rows-transmission}
\begin{equation}
  \mathsf{E}_0\Phi_0=\sum_{i=0}^{2}\Lbf^{0,i}\Phi_i\qquad\text{on }I_0,
\end{equation}
where
\begin{equation}
  \mathsf{E}_0:=\tfrac12\big(\Id+\Neta\big)=\operatorname{diag}\!\Big(1,\tfrac{1+\eta}{2}\Big),
  \qquad
  \Lbf^{0,0}:=\Cop^{0,0}\Neta-\Cop^{0,0}_2 ,
\end{equation}
\end{subequations}
with $\Id$ the identity.

\subsection{Wall rows}
By Lemma~\ref{lem:green-extinction} and \eqref{eq:extension}, the traces of the stack representation
\eqref{eq:rep-stack} of a Bloch mode, taken from inside $S_j$, are
\begin{equation*}
  \bgamma^+_ju_{S_j}[\bphi]=\Phi_j \text{ on } \Gamma_j
  \qquad\text{and}\qquad
  \bgamma^-_j[\aaa_j]u_{S_j}[\bphi]=\zeta_j\Phi_j \text{ on } \Gamma_j+\aaa_j .
\end{equation*}
Each identity alone would give an equation for $\Phi_j$, but both arcs are source curves of
\eqref{eq:rep-stack}, so each trace involves the hypersingular self-interaction $\Wop^{j,j}$ and the
operators $\Cop^{j,j}[\pm\aaa_{j'}]$ of the collinear neighbouring arcs, whose kernels are singular at
the endpoints of $I_j$.  As in \cite[Eqs.~(41)--(42)]{SFFP}, these terms cancel when the second trace
is multiplied by $\zeta_j^{-1}$ and added to the first. In detail, we have:
\begin{proposition}%
\label{prop:wall-rows}
For $j\in\{1,2\}$, $j'$ as in \eqref{eq:jprime}, and $\bphi\in \mathbf{X}$,
\begin{equation}
  \bgamma^+_ju_{S_j}[\bphi]+\zeta_j^{-1}\bgamma^-_j[\aaa_j]u_{S_j}[\bphi]
  =\Phi_j+\sum_{i=0}^{2}\Lbf^{j,i}\Phi_i,
  \label{eq:wall-combination}
\end{equation}
with $u_{S_j}[\bphi]$ defined in~\eqref{eq:rep-stack} and the stack blocks given by 
\begin{equation}
\begin{aligned}
  \Lbf^{j,0}&:=\sum_{n=-1}^{1}\zeta_{j'}^{n}\Big(\Cop^{j,0}[n\aaa_{j'}]
     +\zeta_j^{-1}\Cop^{j,0}[-\aaa_j+n\aaa_{j'}]\Big)\Neta,\\
  \Lbf^{j,j}&:=\sum_{n=-1}^{1}\zeta_{j'}^{n}\Big(\zeta_j^{-1}\Cop^{j,j}[-\aaa_j+n\aaa_{j'}]
     -\zeta_j\Cop^{j,j}[\aaa_j+n\aaa_{j'}]\Big),\\
  \Lbf^{j,j'}&:=\zeta_{j'}^{-1}\Big(\Cop^{j,j'}[-\aaa_{j'}]+\zeta_j^{-1}\Cop^{j,j'}[-\aaa_j-\aaa_{j'}]\Big)\\
     &\qquad-\zeta_{j'}^{2}\Big(\Cop^{j,j'}[2\aaa_{j'}]+\zeta_j^{-1}\Cop^{j,j'}[-\aaa_j+2\aaa_{j'}]\Big).
\end{aligned}
  \label{eq:blocks-stack}
\end{equation}
In every integral operator term of \eqref{eq:blocks-stack} the source curve is disjoint from the target curve~$\Gamma_j$.
\end{proposition}

\begin{proof}
Fix $j\in\{1,2\}$ and $\bphi\in\mathbf X$.  By Lemma~\ref{lem:translation} and the definition of $\Gpot^i$ and $\Cop^{j,i}$ in~\eqref{eq:green-layer}, at interior points of
the target curve $\Gamma_j+\vv'$ we have 
\begin{equation*}
  \bgamma^\pm_j[\vv']\,\Gpot^i[\vv]\Phi
  =\pm\tfrac12\,\delta_{i,j}\,\delta_{\vv,\vv'}\,\Phi+\Cop^{j,i}[\vv-\vv']\Phi ,
\end{equation*}
since every source other than the target intersects it at most at its endpoints, as shown at the end of
the proof.  The two
traces $\bgamma_j^+$ and $\bgamma_j^-$ in \eqref{eq:wall-combination} have $\vv'=\mathbf 0$ and $\vv'=\aaa_j$, so, jumps aside, a
term $\alpha\,\Gpot^i[\vv]$ of \eqref{eq:rep-stack}, where $\alpha$ is its Bloch phase factor (e.g.,
$\alpha=-\zeta_j\zeta_{j'}^n$ for $\Gpot^j[\aaa_j+n\aaa_{j'}]$), contributes
$\alpha\,\big(\Cop^{j,i}[\vv]+\zeta_j^{-1}\Cop^{j,i}[\vv-\aaa_j]\big)$.  For $i=0$ and $i=j'$ this gives
$\Lbf^{j,0}\Phi_0$ and $\Lbf^{j,j'}\Phi_{j'}$ directly.  For $i=j$, the pair of terms with index $n$
gives
\begin{equation*}
  \zeta_{j'}^{n}\Big(\Cop^{j,j}[n\aaa_{j'}]+\zeta_j^{-1}\Cop^{j,j}[-\aaa_j+n\aaa_{j'}]
  -\zeta_j\Cop^{j,j}[\aaa_j+n\aaa_{j'}]-\Cop^{j,j}[n\aaa_{j'}]\Big),
\end{equation*}
in which $\Cop^{j,j}[n\aaa_{j'}]$ cancels, and the sum over $n$ is $\Lbf^{j,j}\Phi_j$.  The jumps
come from the wall terms with $n=0$: $+\frac12\Phi_j$ from $\Gpot^j$ on $\Gamma_j$ ($+$ side) and
$\zeta_j^{-1}(-\zeta_j)(-\frac12)\Phi_j=\frac12\Phi_j$ from $-\zeta_j\Gpot^j[\aaa_j]$ on
$\Gamma_j+\aaa_j$ ($-$ side).  This proves~\eqref{eq:wall-combination}.

It remains to prove that the source curve of every term of \eqref{eq:blocks-stack} is disjoint
from $\Gamma_j$.  Two facts suffice.  First, an inclusion boundary $\Gamma_0+\vv$ lies in the open cell
$U+\vv$, so it intersects no wall arc, which settles~$\Lbf^{j,0}$.  Second, wall arcs are edges of the
tiling $\overline U+\Lambda$, so two distinct wall arcs are disjoint unless they share an endpoint.  For
$\Lbf^{j,j}$ and $\Lbf^{j,j'}$ it therefore suffices to compare endpoints.  Write
$\vv=m\aaa_j+n\aaa_{j'}$.  The target $\Gamma_j$ has endpoints $\mathbf c$, the junction of $\Gamma_1$ and $\Gamma_2$, and
$\mathbf c+\aaa_{j'}$.
In $\Lbf^{j,j}$ the sources are $\Gamma_j+\vv$, with endpoints $\mathbf c+\vv$ and
$\mathbf c+\vv+\aaa_{j'}$, which coincide with an endpoint of $\Gamma_j$ only if $m=0$, whereas $m=\pm1$
in $\Lbf^{j,j}$.  In $\Lbf^{j,j'}$ the sources are $\Gamma_{j'}+\vv$, with endpoints $\mathbf c+\vv$ and
$\mathbf c+\vv+\aaa_j$, which coincide with an endpoint of $\Gamma_j$ only if $n\in\{0,1\}$, whereas
$n\in\{-1,2\}$ in $\Lbf^{j,j'}$.
\end{proof}

It follows from Proposition~\ref{prop:wall-rows} that imposing $\bgamma^+_ju_{S_j}+\zeta_j^{-1}\bgamma^-_j[\aaa_j]u_{S_j}=2\Phi_j$ we obtain the
wall equations
\begin{equation}
  \Phi_j=\sum_{i=0}^{2}\Lbf^{j,i}\Phi_i\qquad\text{on } I_j,\quad j\in\{1,2\}.
  \label{eq:rows-walls}
\end{equation}
By the last statement of Proposition~\ref{prop:wall-rows}, all operators in \eqref{eq:rows-walls} have
kernels smooth up to the endpoints of~$I_j$, so the wall equations are well defined there for every
$\bphi\in\mathbf X$, although the individual traces in \eqref{eq:wall-combination} are in general singular
at the endpoints.

\subsection{The assembled system}

Equations \eqref{eq:rows-transmission} and \eqref{eq:rows-walls} assemble into the following block system for the density vector $\bphi$ of a Bloch mode:
\begin{equation}
  \big(\Ebf-\Lbf\big)\bphi=0,
  \qquad
  \Ebf:=\operatorname{diag}\big(\mathsf{E}_0,\Id,\Id\big),
  \qquad
  \Lbf:=\big[\Lbf^{\ell,i}\big]_{\ell,i=0}^{2}.
  \label{eq:system}
\end{equation}
Here $\mathsf{E}_0$ is defined in \eqref{eq:rows-transmission}, and
the blocks $\Lbf^{0,i}$ are those of \eqref{eq:blocks-transmission} and $\Lbf^{j,i}$, $j\in\{1,2\}$, those of
\eqref{eq:blocks-stack}.  The only one in which target and source curve coincide is $\Lbf^{0,0}$, whose
entries are the four differences
\begin{equation}
  \Lbf^{0,0}=
  \begin{bmatrix}
    \Kop^{0,0}-\Kop^{0,0}_2 & -\big(\eta\Vop^{0,0}-\Vop^{0,0}_2\big)\\[2pt]
    \Wop^{0,0}-\Wop^{0,0}_2 & -\big(\eta\Ktop^{0,0}-\Ktop^{0,0}_2\big)
  \end{bmatrix};
  \label{eq:singular-blocks}
\end{equation}
all other blocks have kernels smooth up to the endpoints, since their source and target curves are
disjoint ($\Gamma_0$ lies in the open cell, and Proposition~\ref{prop:wall-rows}).  The entries of \eqref{eq:singular-blocks} are at most logarithmically singular, owing to the
M\"uller cancellation in $\Wop^{0,0}-\Wop^{0,0}_2$ (see the proof of Proposition~\ref{prop:fredholm}).

The system~\eqref{eq:system} thus has no corner or wall singularities, as a result of three exact
cancellations, namely the M\"uller cancellation on $\Gamma_0$, the cancellation of the wall
self-interaction on each translation pair as in \cite[Eq.~(42)]{SFFP}, and the cancellation of the
interactions with the collinear neighbouring arcs, which the stack representations achieve.

\subsection{Fredholm property}
\label{sec:fredholm}

We here show that $\Ebf-\Lbf$ is Fredholm of index zero and that its null vectors are smooth.  These properties underpin the numerical method of \S\ref{sec:numerics} in three ways.  First, the inverse of $\Ebf-\Lbf$ is meromorphic in $\omega$, with isolated poles at the frequencies where the operator is singular, which justifies approximating the resolvent by rational functions.  Second, each pole carries a nontrivial null vector, which \S\ref{sec:spurious} tests against the Bloch problem. Third, the null vectors are smooth, so a high-order discretization approximates them accurately.

The Fredholm property requires a Banach space, and $\mathbf X$, in which the identities of
\S\ref{sec:bie} were derived, is not one.  We therefore work in the Hilbert space
$\mathbf{Y}:=\prod_{\ell=0}^{2}\big(L^{2}(I_\ell)\times L^{2}(I_\ell)\big)$, on which the operators of \eqref{eq:system} are bounded.  The same arguments apply in other
Banach spaces, for instance spaces of continuous or H\"older
continuous functions.

\begin{proposition}
\label{prop:fredholm}
Assume that the contrast $\eta$ in \eqref{eq:bvp-transmission} satisfies $\eta\ne-1$.  Then, for every $\omega\in\mathbb C\setminus(-\infty,0]$ %
and every $\zeta\in\mathbb T^2$, the operator $\Lbf$ in \eqref{eq:system} is compact on $\mathbf{Y}$ and $\Ebf$ is invertible, so $\Ebf-\Lbf=\Ebf\big(\Id-\Ebf^{-1}\Lbf\big)$ is Fredholm of index zero on $\mathbf{Y}$.  Moreover, every solution $\bphi\in\mathbf{Y}$ of
\eqref{eq:system} is $C^\infty$, i.e.\ belongs to $\mathbf{X}$.
\end{proposition}

\begin{proof}
Every block of \eqref{eq:blocks-transmission} and \eqref{eq:blocks-stack}
other than $\Lbf^{0,0}$ has a source curve at positive distance from the target curve, hence a
kernel that is smooth on the closed parameter rectangle $\overline{I_\ell}\times\overline{I_i}$ of
$\Lbf^{\ell,i}$, and is Hilbert--Schmidt, hence compact~\cite[Ch.~2]{Kress2014}.  In
\eqref{eq:singular-blocks}, by \eqref{eq:kernels}, the kernels of $\Kop^{0,0}_s$ and $\Ktop^{0,0}_s$
are continuous and those of $\Vop^{0,0}_s$ are logarithmically singular.  The kernels of
$\Wop^{0,0}_s$ are hypersingular, but they split as~\cite{Kress1995}
\begin{equation}
  Q^{0,0}_{W,s}(t,\tau)=\frac{\nn_0(t)\cdot\nn_0(\tau)}{2\pi R^2_{0,0}}
  +a_s(t,\tau)\log\Big(4\sin^2\frac{t-\tau}{2}\Big)+b_s(t,\tau),
  \label{eq:hypersing}
\end{equation}
with $a_s,b_s$ smooth and $2\pi$-periodic in both parameters.  The leading term does not depend
on the wavenumber, so it cancels in $Q^{0,0}_{W,1}-Q^{0,0}_{W,2}$, which leaves a logarithmic
singularity (this is the classical M\"uller cancellation \cite{KR1978}).  All kernels of $\Lbf^{0,0}$
are therefore square integrable, so $\Lbf^{0,0}$ is Hilbert--Schmidt as well.
$\Ebf=\operatorname{diag}(1,\tfrac{1+\eta}2,1,1,1,1)$ is invertible because $\eta\ne-1$.

To prove the regularity, let $(\Ebf-\Lbf)\bphi=0$ with $\bphi\in\mathbf{Y}$.  The wall rows
\eqref{eq:rows-walls} read $\Phi_j=\sum_{i}\Lbf^{j,i}\Phi_i$, $j\in\{1,2\}$, and all operators on the right
have kernels in $C^\infty(\overline{I_j}\times\overline{I_i})$, so the wall densities satisfy
$\Phi_j\in C^\infty(\overline{I_j})^2$.  The transmission rows then read
$\mathsf E_0\Phi_0=\Lbf^{0,0}\Phi_0+g$, where $g:=\Lbf^{0,1}\Phi_1+\Lbf^{0,2}\Phi_2$ is smooth because the
kernels of $\Lbf^{0,j}$ are.  On the smooth closed curve $\Gamma_0$ all
entries of $\Lbf^{0,0}$ have kernels of the form $a(t,\tau)\log\big(4\sin^2\frac{t-\tau}2\big)+b(t,\tau)$
with $a,b$ smooth and $2\pi$-periodic (for $\Wop^{0,0}-\Wop^{0,0}_2$ by \eqref{eq:hypersing}), so they map
the periodic Sobolev space $H^r$ into $H^{r+1}$~\cite[Ch.~8]{Kress2014}.  Hence $\Phi_0\in H^r\times H^r$ implies $\Phi_0\in H^{r+1}\times H^{r+1}$, and induction from $r=0$ gives $\Phi_0\in H^r\times H^r$ for every $r\ge0$.  Since $H^r\subset C^k$ for $r>k+\tfrac12$~\cite{McLean2000}, $\Phi_0$ is therefore~$C^\infty$.
\end{proof}

Since $\omega\mapsto(\Ebf-\Lbf)(\omega,\kk)$ is holomorphic on $\mathbb{C}\setminus(-\infty,0]$, the analytic Fredholm theorem~\cite[Thm.~8.26]{CK2012}, applied to $\Id-\Ebf^{-1}\Lbf$, shows that either $\Ebf-\Lbf$ is singular for every $\omega$, or its singular frequencies are discrete, its null spaces are finite-dimensional, and its inverse is meromorphic in $\omega$.  The second case holds if the operator is invertible at a single frequency, which we assume but do not prove.

\section{Null vectors and Bloch modes}
\label{sec:spurious}

The system \eqref{eq:system} was derived from the Cauchy data of a Bloch mode.  This section addresses the converse question of which null vectors of \eqref{eq:system} arise from Bloch modes.  In general not all of them do, and we characterize those that do not, which we call \emph{spurious}.  Throughout, $k_1,k_2,\eta>0$ and $\zeta\in\mathbb T^2$.  We first summarize what the construction of \S\ref{sec:bie} already shows, namely that every Bloch mode yields a null vector.
\begin{lemma}
\label{lem:necessity}
Let $u$ be a Bloch mode, i.e., a nontrivial solution of \eqref{eq:bvp}, and $\bphi$ its Cauchy data
\eqref{eq:densities}.  Then $\bphi\ne0$ and $(\Ebf-\Lbf)\bphi=0$.
\end{lemma}
\begin{proof}
By elliptic regularity for the transmission problem~\cite{McLean2000}, $\bphi\in\mathbf X$, and
$(\Ebf-\Lbf)\bphi=0$ by construction.  Finally, $\bphi\ne0$, since otherwise $u_U[\bphi]$ and $u_\Omega[\bphi]$
vanish, so $u=0$ in the cell by Lemma~\ref{lem:green-extinction}, hence everywhere by
\eqref{eq:bvp-bloch}, contradicting $u\ne0$.
\end{proof}

The converse may fail.  Given a nontrivial solution $\bphi\in\mathbf{X}$ of~\eqref{eq:system}, let
$u_U:=u_U[\bphi]$ and $u_\Omega:=u_\Omega[\bphi]$ be the cell and interior representations
\eqref{eq:rep-cell} and \eqref{eq:rep-interior}, and consider the candidate mode
\begin{equation}
  u:=\mathrm{e}^{i\kk\cdot\vv}u_{\mathrm c}(\cdot-\vv)\ \text{ in } U+\vv,\quad \vv\in\Lambda,
  \qquad\text{where}\quad
  u_{\mathrm c}:=\begin{cases} u_U & \text{in }U\setminus\overline{\Omega},\\[2pt] u_\Omega & \text{in }\Omega.\end{cases}
  \label{eq:candidate}
\end{equation}
We call $\bphi$ \emph{genuine} if it is the Cauchy data of a Bloch mode.  Since the potentials solve the
Helmholtz equation, a spurious $\bphi$ can fail in only two ways, either (T) the transmission
conditions \eqref{eq:bvp-transmission} fail on $\Gamma_0$, or (Q) the Cauchy data of the phased
extension in \eqref{eq:candidate} do not match across $\partial U$.  We show that (T) never occurs
(Lemma~\ref{lem:complementary}) and that a null vector is genuine exactly when its cell representation
vanishes outside the cell (Theorem~\ref{thm:extinction}).  For a spurious null vector, this exterior field is then a nonzero outgoing solution of the Helmholtz equation.
\begin{lemma}
\label{lem:complementary}
Let $\bphi=(\Phi_0,\Phi_1,\Phi_2)\in \mathbf{X}$ satisfy \eqref{eq:rows-transmission}, but not necessarily the
wall equations \eqref{eq:rows-walls}, and let
$u_U:=u_U[\bphi]$ and $u_\Omega:=u_\Omega[\bphi]$ be the cell and interior representations
\eqref{eq:rep-cell} and \eqref{eq:rep-interior}.  Then $\bgamma^+_0u_U=\Neta\Phi_0$,
$\bgamma^-_0u_\Omega=\Phi_0$, and $u_U=0$ in $\Omega$.  In particular, the candidate \eqref{eq:candidate} satisfies the transmission
conditions \eqref{eq:bvp-transmission} on $\Gamma_0$.
\end{lemma}

\begin{proof}
Let $e_0^+:=\bgamma^+_0u_U-\Neta\Phi_0$ and $e_0^-:=\bgamma^-_0u_\Omega-\Phi_0$ be the transmission
residuals.  By the jump relations~\eqref{eq:jumps-pair}, the identities \eqref{eq:traces-transmission} are
$e_0^+=0$ and $e_0^-=0$, and together they give the transmission conditions on $\Gamma_0$ for the
candidate \eqref{eq:candidate}.  The system, however,
imposes only the sum of these two identities, namely \eqref{eq:rows-transmission}, which reads
$e_0^++e_0^-=0$.  The point of this proof is that, for $k_1,k_2>0$, this sum alone forces both residualss to vanish.  The argument is the one used in~\cite{KR1978} to prove the injectivity of M\"uller's BIE system for the Helmholtz transmission problem.

Let $w_1:=u_U|_{\Omega}$, which solves the Helmholtz equation at wavenumber $k_1$ in $\Omega$, and
$w_2:=u_\Omega|_{\mathbb{R}^2\setminus\overline{\Omega}}$, which solves it at $k_2$ outside $\Omega$ and is
radiating, since \eqref{eq:rep-interior} is a free-space potential on the bounded curve $\Gamma_0$.
Across $\Gamma_0$, $u_U$ jumps by $\Neta\Phi_0$ and $u_\Omega$ by $-\Phi_0$, so
$\bgamma^-_0w_1=e_0^+$ and $\bgamma^+_0w_2=e_0^-$.  Now define $v:=w_1$ in $\Omega$ and $v:=-w_2$ in
$\mathbb{R}^2\setminus\overline{\Omega}$.  Its Cauchy data on $\Gamma_0$ are $e_0^+$ from inside and
$-e_0^-$ from outside, and these coincide because $e_0^++e_0^-=0$.  Thus $v$ is a radiating solution
of the transmission problem with wavenumber $k_1$ in $\Omega$, $k_2$ outside, and continuous Cauchy
data across $\Gamma_0$.  Since $k_1,k_2>0$, this problem has only the trivial
solution~\cite[Thm.~3.40]{CK1983}, so $v=0$.  Hence $u_U=w_1=0$ in $\Omega$, $e_0^+=\bgamma^-_0w_1=0$,
and $e_0^-=-e_0^+=0$.
\end{proof}

Next we turn our attention to (Q). For a Bloch mode, Green's formula makes the cell representation \eqref{eq:rep-cell} vanish outside the cell (Lemma~\ref{lem:green-extinction}).  This property characterizes genuine null vectors.

\begin{theorem}%
\label{thm:extinction}
Let $\bphi=(\Phi_0,\Phi_1,\Phi_2)\in \mathbf{X}$ satisfy \eqref{eq:rows-transmission}, but not necessarily the
wall equations \eqref{eq:rows-walls}, and let $w:=u_U[\bphi]|_{\Ue}$ be the
restriction of the cell representation~\eqref{eq:rep-cell} to the exterior
$\Ue:=\mathbb{R}^2\setminus\overline{U}$ of the cell.  Then $w\equiv0$ in $\Ue$ if and only if the candidate~\eqref{eq:candidate} is a
quasi-periodic transmission solution with Cauchy data~$\bphi$.  In particular, a nontrivial null vector of \eqref{eq:system} is genuine if and only if $w\equiv0$.
\end{theorem}

\begin{proof}
If the candidate $u$ in~\eqref{eq:candidate} solves \eqref{eq:bvp} with Cauchy data $\bphi$, in the sense of
\eqref{eq:densities}, then $w\equiv0$
by Lemma~\ref{lem:green-extinction}.

Conversely, suppose $w\equiv0$, and write $u_U:=u_U[\bphi]$.  Then $u_U=0$ in $\Ue$, and in $\Omega$ by
Lemma~\ref{lem:complementary}.  We show that the phased lattice sum
$\tilde u:=\sum_{\vv\in\Lambda}\mathrm{e}^{i\kk\cdot\vv}u_U(\cdot-\vv)$, which coincides with the candidate in
$D_1=\mathbb R^2\setminus\overline{D_2}$, extends smoothly across the walls, with Cauchy data $\Phi_j$ on
them.  The term $\vv$ of $\tilde u$ vanishes outside $(\overline U\setminus\Omega)+\vv$, so near any point
at most four terms are nonzero.  For $\mathbf a\in\Lambda$, substituting $\vv=\vv'+\mathbf a$ in the sum gives
$\tilde u(\cdot+\mathbf a)=\sum_{\vv'\in\Lambda}\mathrm e^{i\kk\cdot(\vv'+\mathbf a)}u_U(\cdot-\vv')
=\mathrm e^{i\kk\cdot\mathbf a}\tilde u$, so $\tilde u$ is quasi-periodic, and since it
equals $u_U$ in $U\setminus\overline\Omega$, it coincides with the candidate in $D_1$.

Near any point of $D_1$, walls and corners included, $\tilde u$ is the sum of the terms of the cells
whose closures contain the point.  In this finite sum, each wall arc $\Gamma_j+\vv$, $j\in\{1,2\}$,
through the point appears in the terms of the two cells it separates, $U+\vv$ and $U+\vv-\aaa_j$, with
densities $\mathrm{e}^{i\kk\cdot\vv}\Phi_j$ and $-\zeta_j\mathrm{e}^{i\kk\cdot(\vv-\aaa_j)}\Phi_j$ by
\eqref{eq:rep-cell} and \eqref{eq:shift}, which cancel since $\zeta_j=\mathrm{e}^{i\kk\cdot\aaa_j}$.  Near the
point, $\tilde u$ is therefore a finite sum of layer potentials on curves away from it, which gives the
smooth extension.  On $\Gamma_j$ the normal $\nn_j$ points into $U$ (Figure~\ref{fig:geometry}), so
by~\eqref{eq:jumps-pair} $\bgamma^+_ju_U-\bgamma^-_ju_U=\Phi_j$ with $\bgamma^-_ju_U=0$, and the trace of
$\tilde u$ on $\Gamma_j$ is $\bgamma^+_ju_U=\Phi_j$.  Finally, by Lemma~\ref{lem:complementary}, the
candidate satisfies the transmission conditions with Cauchy data $\Phi_0$ on $\Gamma_0$, so it is a
quasi-periodic transmission solution with Cauchy data $\bphi$. This completes the proof.
\end{proof}

In practice, to verify that a null vector is genuine, it suffices to check that $w$ vanishes on a closed curve around the cell, as the following corollary shows.  Since $w$ is real analytic in the
connected set $\Ue$, other conditions would serve as well, for instance $w=0$ on a nonempty open subset of $\Ue$, or the vanishing of both $w$ and its normal derivative on an arc.
\begin{corollary}
\label{cor:curve}
Let $\bphi$ be a nontrivial null vector of \eqref{eq:system}, and let $\mathcal C$ be a smooth closed curve
enclosing $\overline U$.  Then $\bphi$ is genuine if and only if $u_U[\bphi]=0$ on $\mathcal C$.
\end{corollary}

\begin{proof}
Let $w:=u_U[\bphi]|_{\Ue}$ as the theorem states.  If $\bphi$ is genuine, then $w\equiv0$ by Theorem~\ref{thm:extinction}, so in
particular $w=0$ on $\mathcal C$.  Conversely, suppose $w=0$ on $\mathcal C$.  Since $u_U[\bphi]$ is a finite
sum of layer potentials with the free-space kernel on bounded curves, $w$ is an outgoing solution of
the Helmholtz equation at wavenumber $k_1$ in $\Ue$.  Green's theorem in the region between $\mathcal C$
and a circle $|\rr|=R$ enclosing it gives $\operatorname{Im}\int_{|\rr|=R}\overline w\,\partial_rw\,\mathrm ds=0$,
so $w=0$ outside $\mathcal C$ by Rellich's lemma~\cite[Thm.~2.14]{CK2012}, and in all of $\Ue$ by unique
continuation.  Hence $\bphi$ is genuine by Theorem~\ref{thm:extinction}.
\end{proof}

\begin{remark}
By Theorem~\ref{thm:extinction}, extinction outside the cell could replace the wall equations in
\eqref{eq:system}.  Imposing it, however, means evaluating the potentials away from their sources, a
smoothing operation that destroys the second-kind structure.  We therefore use extinction only as an
a posteriori test.
\end{remark}

Lemmas~\ref{lem:necessity} and~\ref{lem:complementary}, Theorem~\ref{thm:extinction} and
Corollary~\ref{cor:curve} give the picture used in computations.  No Bloch mode is missed, and a
spurious null vector is recognized by a nonzero cell representation on a curve around the cell.  The
field residuals of \S\ref{sec:screening} sample it at finitely many points and are a numerical
surrogate for this test.  

\section{Discretization and band search}
\label{sec:numerics}

\subsection{Discretization}
\label{sec:discretization}

We discretize the BIE system~\eqref{eq:system} by a Nystr\"om method with composite Gauss--Legendre quadrature on the inclusion boundary and the two walls, implemented in the Julia package \texttt{Inti.jl}~\cite{Inti,AFP2024}.  Singular and nearly singular integrals, in the operators as well as in the evaluation of the layer potentials near the curves, are computed by the general-purpose density interpolation method~\cite{FPB2021}, which allows for closely spaced curves to be handled without local refinement. (On the closed curve $\Gamma_0$, the logarithmically singular self-interaction could instead be discretized by Martensen--Kussmaul quadrature~\cite[\S3.5]{CK2012}, which is spectrally accurate for smooth curves.  The walls are open arcs and would still require a separate treatment.)  Translated curves inherit the quadrature nodes of the original curves, so the cancellations of Proposition~\ref{prop:wall-rows} are exact at the discrete level.

It follows from \eqref{eq:blocks-transmission} and \eqref{eq:blocks-stack} that the BIE operator is a Laurent polynomial in the Bloch multipliers $\zeta_1,\zeta_2$,
\begin{equation}
  \Ebf-\Lbf=\sum_{(p,q)\in\mathcal{I}}\zeta_1^{\,p}\zeta_2^{\,q}\,\mathsf{T}^{(p,q)}(\omega),
  \qquad
  \mathcal{I}=\{-1,0,1\}^2\cup\{(-1,2),(0,2),(2,-1),(2,0)\},
  \label{eq:polystructure}
\end{equation}
with thirteen operator coefficients $\mathsf T^{(p,q)}(\omega)$.  Each term collects the operators $\Cop^{\ell,i}[\vv]$ (and, for $(p,q)=(0,0)$, $\Ebf$) that multiply $\zeta_1^p\zeta_2^q$ in \eqref{eq:blocks-transmission} and \eqref{eq:blocks-stack}. These depend on $\omega$, through $k_1,k_2$, and on the geometry, but not on $\kk$, which enters only through the scalar weights $\zeta_1^p\zeta_2^q$, with $\zeta_j=\mathrm{e}^{i\kk\cdot\aaa_j}$ as in~\eqref{eq:zeta-def}.  Many of the coefficients are moreover block sparse (Figure~\ref{fig:sparsity}).

The discretization preserves this structure.  If $T^{(p,q)}(\omega)\in\mathbb C^{N\times N}$ denotes the
Nystr\"om discretization of $\mathsf T^{(p,q)}(\omega)$, with $N$ the total number of unknowns, the
discretization of $\Ebf-\Lbf$ is given by
\begin{equation}
  T(\omega,\kk)=\sum_{(p,q)\in\mathcal I}\zeta_1^{\,p}\zeta_2^{\,q}\,T^{(p,q)}(\omega).
  \label{eq:polystructure-discrete}
\end{equation}

At each frequency $\omega$, the thirteen matrices $T^{(p,q)}(\omega)$ are assembled once, compressed as
hierarchical matrices ($\mathcal H$-matrices)~\cite{Hackbusch2015,HMatrices}, and reused for every Bloch vector $\kk$, in particular for all points of the Brillouin-zone path along which a band diagram is computed~(\S\ref{sec:continuation}).

\begin{figure}[tb]
\centering
\includegraphics{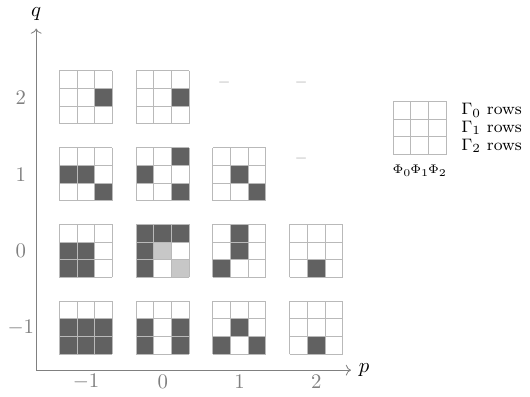}
\caption{Block sparsity of the thirteen coefficients $\mathsf T^{(p,q)}$ of \eqref{eq:polystructure}.
At position $(p,q)$ of the grid, a $3\times3$ array marks the nonzero blocks of $\mathsf T^{(p,q)}$, with rows indexed by the equations on $\Gamma_0,\Gamma_1,\Gamma_2$ and columns by $\Phi_0,\Phi_1,\Phi_2$.  Dark cells come from the operator
$\Lbf$, light cells from the identity blocks of $\Ebf$, and white cells are zero blocks. A dash marks
an exponent $(p,q)\notin\mathcal I$.}
\label{fig:sparsity}
\end{figure}

\subsection{Computing bands by AAA rational approximation}
\label{sec:aaa}

We use the randomly scalarized resolvent method put forth in~\cite{BST2026} to locate the real frequencies at which~$T(\omega,\kk)$ is singular.  We summarize it below and specify its use with our stack formulation, in particular the reuse of assembled operators across Bloch vectors via~\eqref{eq:polystructure} and the screening of candidate modes by the field residuals of \S\ref{sec:screening}.

\subsubsection{Detection and refinement}
\label{sec:sampling}
For a fixed Bloch vector $\kk$, we sample the scalarized resolvent
\begin{equation}
 f_\kk(\omega)=\mathbf p^{\mathsf H}T(\omega,\kk)^{-1}\mathbf q
 \label{eq:scalarized}
\end{equation}
at real frequencies $\omega$, where $\mathbf p,\mathbf q\in\mathbb C^{N}$ are random probe vectors.  If the
matrix family $T(\cdot,\kk)$ is holomorphic and regular, random scalarization preserves its poles with
probability one~\cite{BST2026}.  Each sample costs one unrestarted GMRES solve~\cite{SS1986} with the
compressed operators of \S\ref{sec:discretization}.  The poles near the real axis of AAA rational
approximants~\cite{AAA2018}, computed with \texttt{Rational\allowbreak Function\allowbreak Approximation.jl}~\cite{RFA}, of $f_\kk$ are
candidates for Bloch frequencies, which are then refined and screened.
The scalar fit does not give the multiplicity of a Bloch frequency, which a minor modification of the
algorithm would provide~\cite{BST2026}.  We do not pursue this here.

We use smooth probes.  The density interpolation method attains high order by interpolating the
density at each target point, so it is accurate only for smooth densities, and the solutions
$T(\omega,\kk)^{-1}\mathbf q$ of a second-kind system are only as smooth as $\mathbf q$.  On each of the six
blocks of unknowns, a probe is the restriction to the quadrature nodes of a random field
$\sum_{m=1}^{16}c_m e^{\mathrm i\boldsymbol\kappa_m\cdot\mathbf x}$, with complex Gaussian coefficients $c_m$ and
wave vectors $\boldsymbol\kappa_m$ of uniformly random direction and length uniform in $(0,8\pi/a]$, that is,
wavelengths of at least $a/4$.  Both probes are normalized to Euclidean norm $\sqrt N$ and kept fixed
within a search.

We use the adaptive search of~\cite[Alg.~2]{BST2026}, in the form implemented by its authors for real
intervals.  AAA is first applied to equispaced samples on a prescribed positive-frequency interval,
endpoints included, and the interval is then bisected once.  Each subinterval keeps the parent
samples that lie in it and adds the midpoints between consecutive ones.  The sample sets are thus
nested, sample values are cached, and each subinterval needs only about half as many new samples as its
parent.  The poles of the fit on a subinterval are accepted if their number equals that of the parent
fit's poles in the subinterval, or if there are none.  Otherwise the subinterval is bisected again, up
to a maximum depth.  Poles with residue below a small multiple of the median absolute sample value are
discarded as Froissart doublets, spurious pole--zero pairs that the fit introduces to absorb
noise~\cite[\S5]{AAA2018}.  We use the median rather than the maximum, because a sample that falls on a
Bloch frequency has a huge value, against which every genuine pole would look negligible.

Each candidate is refined by the secant method on $g=1/f_\kk$,
\begin{equation}
 \omega^{(n+1)}=\omega^{(n)}-g(\omega^{(n)})
 \frac{\omega^{(n)}-\omega^{(n-1)}}{g(\omega^{(n)})-g(\omega^{(n-1)})},
 \label{eq:secant}
\end{equation}
with the iterates projected onto the real axis.  The iteration stops when $|\omega^{(n+1)}-\omega^{(n)}|$
falls below a tolerance, and this last correction estimates the error of the refined frequency.  The
candidate is accepted if the iteration converges without leaving a prescribed neighbourhood of it.  Close
to a Bloch frequency, $T(\omega,\kk)$ is nearly singular, and GMRES may stagnate before reaching its
tolerance~\cite{BW1997}, although its iterate already points along the null vector.  Inside the secant
iteration, GMRES is therefore capped at three times the median iteration count of the sampling
solves, and at least $60$ iterations, well below the cap used for sampling.  The density
vector~$\bphi_h$ of the resulting mode is the solution of $T(\omega,\kk)\bphi_h=\mathbf q$ at the refined
frequency, one step of inverse iteration, and it is assessed through the extinction
residual~\eqref{eq:residual} below rather than through the convergence of the solve.

Degeneracies and near-degeneracies are generic in band structures, whereas close poles are not
considered in the resonance computations of~\cite{BST2026}.  Two Bloch frequencies much closer together
than the sample spacing may appear as a single pole.  The parent
and subinterval fits then both have one pole there, the acceptance criterion of the adaptive search is
met, and the subinterval is accepted with one frequency missing.  To address this issue, we  add a local check after refinement.  Around each refined frequency $\omega$, we
sample $f_\kk$ at a few equispaced points of $[(1-\delta)\omega,(1+\delta)\omega]$, whose spacing is
comparable to the separations to be resolved, and fit these samples by AAA.  A pole of this local fit
that does not coincide with a frequency already found is refined by the secant method, added, and
checked in the same way.

\subsubsection{Field residuals}
\label{sec:screening}
Each refined candidate comes with a density vector~$\bphi_h$, which we test with three field residuals
that vanish for an exact Bloch mode.  Each residual is the largest absolute value of a field at a few
points of a circle, relative to $\|\bphi_h\|_\infty$.  The extinction residual
\begin{equation}
 \mathrm{res}(\bphi_h)=\max_m|u_U[\bphi_h](\mathbf y_m)|/\|\bphi_h\|_\infty
 \label{eq:residual}
\end{equation}
uses the cell representation~\eqref{eq:rep-cell} at points $\mathbf y_m$ on a circle enclosing $\overline U$.
By Corollary~\ref{cor:curve}, it is small for a genuine mode and of order one for a spurious null
vector.  The other two residuals use $u_U[\bphi_h]$ at points on a circle inside $\Omega$ and
$u_\Omega[\bphi_h]$ at points on a circle around $\Omega$.  These fields vanish for every null vector by
Lemma~\ref{lem:complementary}, and their residuals measure how accurately $\bphi_h$ satisfies the
transmission conditions on $\Gamma_0$, which the extinction residual alone does not detect.  All sample points are away from the source curves, so smooth quadrature suffices.  Since
the corollary and the lemma concern whole curves, the sampled residuals are numerical diagnostics.
Accepted modes are evaluated using~\eqref{eq:rep-cell} in $U\setminus\overline{\Omega}$
and~\eqref{eq:rep-interior} in $\Omega$.

\subsubsection{Assembly reuse along band paths}
\label{sec:continuation}
The Laurent-polynomial structure~\eqref{eq:polystructure-discrete} lets all Bloch vectors of a band path
share assemblies.  The adaptive search of \S\ref{sec:sampling} is run simultaneously for the $K$ path
points, with sample frequencies common to all of them.  At each sample frequency, the thirteen matrices
$T^{(p,q)}(\omega)$ of \S\ref{sec:discretization} are combined with the phase weights of each Bloch
vector without forming a full matrix.  Each Bloch vector keeps its own AAA fits and bisection, so
$M_\omega$ sample frequencies require $M_\omega$ assemblies and at most $KM_\omega$ solves.  Secant
refinement~\eqref{eq:secant} is not shared, since its iterates differ from one Bloch vector to the next,
and costs a few assemblies per Bloch frequency and path point.  Most of this work can run in parallel,
since the sample frequencies of one level of the search, the $K$ solves at one sample frequency, and
the refinements and local checks of different Bloch frequencies are independent.  The timings of
\S\ref{sec:experiments} are for a sequential implementation.

Bands are not traced as curves.  Continuing individual modes along the path, with the derivative of
$T$ in $\kk$ obtained from~\eqref{eq:polystructure-discrete} without new assemblies, lost branches at
degeneracies in our tests.  At each path point the refined frequencies are sorted,
and the $n$th band is the $n$th frequency.  This labelling is unambiguous except at band
crossings and degeneracies, which occur at isolated path points.  At the high-symmetry
points a degenerate frequency appears only once (\S\ref{sec:sampling}), and it is labelled by the nearest band at the neighbouring path
point.  The gaps between consecutive bands, usually the main figure of merit, depend only on the sorted frequencies and are therefore unaffected by crossings.

\section{Numerical experiments}
\label{sec:experiments}

We describe two validation examples with exact or published reference values, and a final set of experiments on a crystal whose cell needs curved walls and nearly singular integrals.  Unless stated otherwise $D_1$ is vacuum, and all media are nonmagnetic, $\mu_1=\mu_2$.  The lattices are square, except the triangular lattice of \S\ref{sec:ex-triangular}, with lattice constant $a$.  Frequencies are reported as $\omega a/(2\pi c)$, where $c$ is the speed of light in vacuum.  In E-polarization $u=E_z$ and $\eta=\mu_1/\mu_2=1$, and in H-polarization $u=H_z$ and $\eta=\varepsilon_1/\varepsilon_2$, as in~\eqref{eq:bvp}; \cite{JJWM2008} calls them TM and TE.  For each experiment we report, where computed, the convergence of the frequencies in the discretization parameters, the extinction residual \eqref{eq:residual} of the computed modes, and the number of
assembled frequencies.

\paragraph{Implementation details.}
The discretization of \S\ref{sec:discretization} uses a quasi-uniform subdivision of each curve, of mesh size $h$, with $q$-point Gauss--Legendre rules, giving $N$ unknowns in total.  The $\mathcal H$-matrices are built by \texttt{HMatrices.jl}~\cite{HMatrices} with tolerance $10^{-8}$, GMRES has relative tolerance $10^{-10}$, and AAA is taken from \texttt{Rational\-Function\-Approximation.jl}. Unless stated otherwise, $h=0.05a$ and $q=9$.

The band search of \S\ref{sec:sampling} uses $33$ equispaced samples per interval, that is, $32$ equal
subintervals, and at most eight bisection levels.  The authors of~\cite{BST2026} recommend at least
$200$ samples per interval for the AAA tolerance $10^{-13}$, but our samples are only as accurate as the
hierarchical compression, and $33$ samples, with the local check, found every Bloch frequency in the
validation examples of \S\ref{sec:ex-bench}.  The crystal of \S\ref{sec:ex-triangular}, with more and
closer bands, needs $65$.

Secant refinement stops when the correction falls below $10^{-10}\max(1,\omega a/c)$, after at most
eight steps.  The local check uses $\delta=10^{-3}$ and $10$ equispaced samples, an even number so that
the refined frequency itself is not sampled.  The sample spacing is then about $2\times10^{-4}\omega$, so
Bloch frequencies closer than about $10^{-4}$ in relative terms may still be merged.

Band diagrams follow the usual convention~\cite{JJWM2008} and plot the Bloch frequencies along the
boundary of the irreducible Brillouin zone, the smallest part of the first Brillouin zone from which the
whole band structure follows by the symmetries of the lattice and the inclusion.  The extrema of the
bands, and hence the gap edges, usually lie on this boundary.  Its vertices are high-symmetry points,
labelled $\Gamma$, X, M for the square lattice and $\Gamma$, M, K for the triangular lattice.  Here
$\Gamma$ is $\kk=0$, and for the square lattice $\mathrm X=(\pi/a,0)$ and $\mathrm M=(\pi/a,\pi/a)$.  Those
of the triangular lattice are given in \S\ref{sec:ex-triangular}.  A band diagram shows $\omega$ against
arc length along the closed path $\Gamma$--X--M--$\Gamma$, with the vertices marked.  For the square
lattice, each of the three segments is divided into $10$ intervals, giving $30$ distinct Bloch vectors.
The validation for the rod lattice at a single Bloch vector uses the generic point
$\kk a/\pi=(0.7,0.3)$.

\paragraph{Measuring accuracy}
Band frequencies are the quantity of interest, and we assess them in four ways.  First, we compare
with reference values, namely the folded cones~\eqref{eq:empty-lattice} for the empty lattice, and the
band edges and gap sizes published in~\cite{JJWM2008} for the rod lattice and for the triangular lattice
of \S\ref{sec:ex-triangular}.  Second, we check self-convergence as the mesh size $h$ and the quadrature
order $q$ are refined.  Third, we report the field residuals of \S\ref{sec:screening}, which assess the
computed density rather than the error in $\omega$.  Fourth, we report the last secant
correction~\eqref{eq:secant}, which~\cite[\S5]{BST2026} use as an error estimate.  It measures the
distance to a root of the discrete system, not the discretization error, and the empty lattice shows
that it can underestimate the actual error. 

\subsection{Validation examples}
\label{sec:ex-bench}

\paragraph{Empty lattice} We first consider a homogeneous medium, $\varepsilon_2=\varepsilon_1$ and $\mu_2=\mu_1$, and keep
$\Gamma_0$ as an artificial interface whose Cauchy data $\Phi_0$ remain unknowns.  We use the square
lattice $\aaa_1=a(1,0)$, $\aaa_2=a(0,1)$, with $\Gamma_0$ the circle of radius $0.2a$, $h=0.05a$ and $q=9$
($N=660$).  The exact photonic bands are the folded light cones~\cite{AM1976,JJWM2008}
\begin{equation}
 \omega(\kk)=c\,\big|\kk+\mathbf{G}\big|,\qquad \mathbf{G}=\tfrac{2\pi}{a}(m,n),\quad m,n\in\mathbb Z,
 \label{eq:empty-lattice}
\end{equation}
where $\mathbf G$ ranges over the reciprocal lattice $\Lambda^*$.  Since $\eta=1$ and $k_2=k_1$, the M\"uller
block $\Lbf^{0,0}$ of~\eqref{eq:rows-transmission} vanishes and $\mathsf E_0=\Id$, so the transmission rows
reduce to $\Phi_0=\sum_j\Lbf^{0,j}\Phi_j$.  For the Cauchy data of an exact Bloch mode, a plane wave
$\mathrm e^{i(\kk+\mathbf G)\cdot\rr}$, the inclusion terms of \eqref{eq:rep-stack} vanish outside $\Omega$ by
Green's formula, but the coupling blocks $\Lbf^{0,j}$ and $\Lbf^{j,0}$ do not.  The example therefore
tests, against an exact reference, the Bloch phases, the orientation signs of \eqref{eq:rep-cell} and
\eqref{eq:rep-stack}, the stack blocks \eqref{eq:blocks-stack} and the coupling between the walls and the
inclusion boundary, but not the M\"uller block, which the next example tests.

\begin{figure}[h!]
\centering
\raisebox{-0.5\height}{\includegraphics{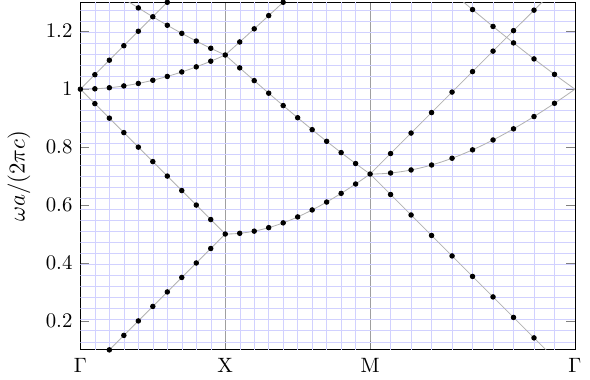}}
\hspace{0.03\textwidth}
\raisebox{-0.5\height}{\includegraphics{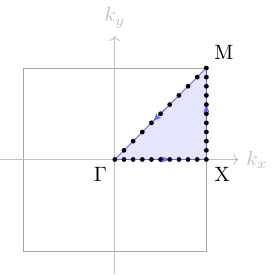}}

\caption{Empty lattice example at $h=0.05a$ and $q=9$ ($N=660$).  Left: exact folded light cones \eqref{eq:empty-lattice}
(lines) and computed bands (dots) against arc length along $\Gamma$--X--M--$\Gamma$, over the
first-level sample grid of the search algorithm (faint lines).  Right: Brillouin zone and irreducible zone (shaded) and the $30$ Bloch vectors of the path.}
\label{fig:empty-lattice}
\end{figure}

We search the frequency window $\omega a/(2\pi c)\in[0.1,1.3]$ along $\Gamma$--X--M--$\Gamma$, with $10$
intervals per segment.  Some exact frequencies lie on the ends of this window, where the first-level
grid places samples, so the search is run on the slightly larger window $[0.09,1.31]$.  With $97$
assembled frequencies, the procedure of \S\ref{sec:aaa} recovers all $93$ distinct exact frequencies at
the $30$ Bloch vectors, up to five at a single Bloch vector, and no spurious frequency.
Figure~\ref{fig:empty-lattice} compares them with \eqref{eq:empty-lattice}, over the first-level grid
of the search, $33$ equispaced frequencies at each of the $30$ Bloch vectors.  The computed frequencies
have absolute error at most $2.1\times10^{-9}$ in $\omega a/c$, with median $1.3\times10^{-10}$, and relative
error at most $6.0\times10^{-10}$, while the AAA poles before refinement are off by up to
$1.1\times10^{-5}$.  The last secant correction \eqref{eq:secant} is at most $7.4\times10^{-10}$, and at
$70$ of the $93$ frequencies it is smaller than the actual error, by up to a factor of about $2\times10^2$.
The secant correction measures the convergence of the iteration to a root of the discrete system,
not the discretization error of the BIE system, which dominates here. We also test three generic Bloch vectors off the symmetry lines, $\kk a/\pi=(0.7,0.3)$, $(0.4,0.1)$ and $(0.9,0.6)$, at which no two exact frequencies coincide.  On the window $\omega a/(2\pi c)\in[0.16,1.27]$,
the same procedure, with $65$ assembled frequencies, finds all $15$ exact frequencies ($5$, $6$ and $4$, respectively) and no spurious one, with absolute errors at most $6.5\times10^{-10}$ in $\omega a/c$.

\paragraph{Square lattice of dielectric rods} Our second example is the square lattice of circular rods of radius $r=0.2a$ and permittivity
$\varepsilon_2=8.9$ in vacuum, in E-polarization, the standard example of \cite[Ch.~5]{JJWM2008}, whose
first two bands are separated by a wide gap.  Unlike the empty lattice, it involves the M\"uller block
on $\Gamma_0$.

We first validate the solver at the generic Bloch vector $\kk a/\pi=(0.7,0.3)$, in the window
$\omega a/c\in[1,6]$, which contains six Bloch frequencies, with $q=9$ and the two mesh sizes $h=0.05a$
and $h=0.025a$.  We assess accuracy by self-convergence, taking the refined frequencies at $h=0.025a$
as reference, and by the field residuals of \S\ref{sec:screening}.  Table~\ref{tab:validation-refined}
gives the refined frequencies at both resolutions.  They agree to $3.8\times10^{-8}$ in $\omega a/c$,
whereas the last secant corrections are at most $2\times10^{-10}$, so, as for the empty lattice, the
secant correction does not estimate the discretization error.  The field residuals are evaluated at
$10$ points on each of three circles, one enclosing the cell for $u_U$, one of radius $0.1a$ inside the
rod for $u_U$, and one of radius $0.3a$ around the rod for $u_\Omega$.  They are at most
$1.2\times10^{-6}$ at $h=0.05a$ and $2.4\times10^{-8}$ at $h=0.025a$, and thus decrease with $h$ like the
frequency error. 

\begin{table}[tb]
\centering\small
\begin{tabular}{cccccc}
  \toprule
  & \multicolumn{2}{c}{$\omega a/(2\pi c)$} & \multicolumn{3}{c}{field residuals, $h=0.025a$} \\
  \cmidrule(lr){2-3}\cmidrule(lr){4-6}
  band & $h=0.05a$ & $h=0.025a$ & $u_U$ outside $U$ & $u_U$ in $\Omega$ & $u_\Omega$ outside $\Omega$ \\
  \midrule
  1 & $0.2454741960$ & $0.2454741910$ & $3.7\times10^{-10}$ & $3.9\times10^{-10}$ & $5.4\times10^{-10}$ \\
  2 & $0.4871327640$ & $0.4871327625$ & $2.3\times10^{-9}$  & $5.3\times10^{-9}$  & $2.8\times10^{-9}$ \\
  3 & $0.6101824436$ & $0.6101824422$ & $4.7\times10^{-10}$ & $1.0\times10^{-9}$  & $5.8\times10^{-10}$ \\
  4 & $0.7306238758$ & $0.7306238765$ & $2.2\times10^{-10}$ & $4.2\times10^{-10}$ & $2.5\times10^{-10}$ \\
  5 & $0.8523783825$ & $0.8523783886$ & $1.3\times10^{-8}$  & $2.4\times10^{-8}$  & $1.4\times10^{-8}$ \\
  6 & $0.9344449188$ & $0.9344449203$ & $9.9\times10^{-10}$ & $1.7\times10^{-9}$  & $1.3\times10^{-9}$ \\
  \bottomrule
\end{tabular}
\caption{Bloch frequencies after secant refinement, $q=9$ ($N=660$ and $1310$), for the
square lattice of rods, $r=0.2a$, $\varepsilon_2=8.9$, at $\kk a/\pi=(0.7,0.3)$, and the field residuals of \S\ref{sec:screening} of
the modes at $h=0.025a$, relative to $\|\bphi_h\|_\infty$.}
\label{tab:validation-refined}
\end{table}

Timings are wall-clock on an Apple M3 Max, with one Julia thread and multithreaded BLAS.  With $N=660$,
assembling the operators at one frequency, compression included, takes $0.10$\,s, and one GMRES solve
takes $0.05$\,s and $18$ iterations, or $19$ at a distance $10^{-4}$ from a Bloch frequency.  The
iteration count grows closer to a Bloch frequency.  At $\kk a/\pi=(0.02,0)$ on
$\omega a/(2\pi c)\in[0.2,1.1]$, GMRES takes $15$ to $36$ iterations (median $22$) at the $97$ samples of
the search, and up to $60$ (median $35$) during refinement, which samples much closer to the Bloch
frequencies.  The complete computation of Figure~\ref{fig:empty-lattice}, with search,
refinement and local checks at the $30$ path points, took $348$\,s.

Figure~\ref{fig:band-diagram} shows the bands of the rod lattice along $\Gamma$--X--M--$\Gamma$, with
$10$ intervals per segment, in the window $\omega a/(2\pi c)\in[0.01,1.1]$, searched on the slightly
larger window $[0.005,1.11]$.  The search, shared by the $30$ distinct Bloch vectors, starts from $33$ equispaced
frequencies, shown in the figure, and assembled $97$ frequencies in total, and refinement with the local check returned $232$ Bloch frequencies, between $5$ and $9$
per Bloch vector, in $948$\,s in total.  The first E-polarization gap lies between the first band at M, $0.32239468$, and the
second band at X, $0.44251872$.  Its gap-to-midgap ratio, the gap width divided by its central frequency, $31.4\%$, agrees with the value
reported
in~\cite[Ch.~5, p.~69]{JJWM2008}, where E-polarization is called TM.

\begin{figure}[tb]
\centering
\includegraphics{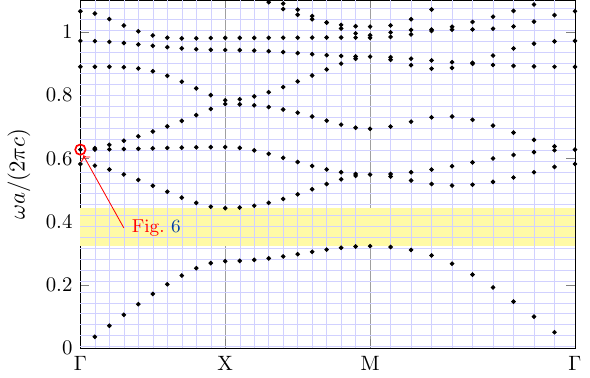}
\caption{Square lattice of rods example, $r=0.2a$, $\varepsilon_2=8.9$, E-polarization, at $h=0.05a$ and
$q=9$ ($N=660$).  Computed bands (dots) against arc length along $\Gamma$--X--M--$\Gamma$, at the Bloch
vectors of Figure~\ref{fig:empty-lattice}, over the first-level sample grid of the search algorithm
(faint lines).  The shaded band is the first gap, and the acoustic band at $\Gamma$ lies below the
window.  The red circle marks the close pair at $\kk a/\pi=(0.003,0)$ of
Figure~\ref{fig:search-anatomy}.}
\label{fig:band-diagram}
\end{figure}

At $\Gamma$ the third and fourth bands are degenerate, at $\omega a/(2\pi c)\approx0.628$, and along
$\Gamma$--X they separate as $|\kk|$ grows.  Such close pairs are hard for the search of
\S\ref{sec:sampling}, since two poles much closer together than the sample spacing look like one.  We
test this at Bloch vectors $\kk a/\pi=(t,0)$ on $\Gamma$--X close to $\Gamma$, on the window
$\omega a/(2\pi c)\in[0.2,1.1]$.  With the local check, the search finds both frequencies for $t=0.01$,
$0.005$, $0.003$ and $0.002$, where they are $4.9\times10^{-5}$, $1.2\times10^{-5}$, $4.5\times10^{-6}$ and
$2.0\times10^{-6}$ apart in $\omega a/(2\pi c)$, but only one of them for $t=0.001$, where they are about
$5\times10^{-7}$ apart.

Figure~\ref{fig:search-anatomy} shows the search at $t=0.003$, marked in Figure~\ref{fig:band-diagram},
where the pair is $0.6278020876$ and $0.6278065385$ and the window contains six Bloch frequencies.  The
search bisects the window once and assembles $65$ frequencies.  Every fit has a single pole near the
pair, at $0.6278026259$, from which secant refinement converges to one of the two frequencies, and the
local check, with ten samples on $\omega(1\pm10^{-3})$, finds the other.  For comparison, the figure also
shows the two smallest singular values $\sigma_1\le\sigma_2$ of the dense matrix $T(\omega,\kk)$.  At this
resolution $\sigma_1$ has a single minimum at the pair, and only $\sigma_2$ reveals a second Bloch
frequency, dropping there to $1.2\times10^{-5}$, the size of $\sigma_1$, against about $5\times10^{-3}$
at the other frequencies.  The inset, computed from $201$ dense decompositions within
$1.8\times10^{-5}$ of the pair, resolves the two minima of $\sigma_1$.

\begin{figure}[tb]
\centering
\includegraphics{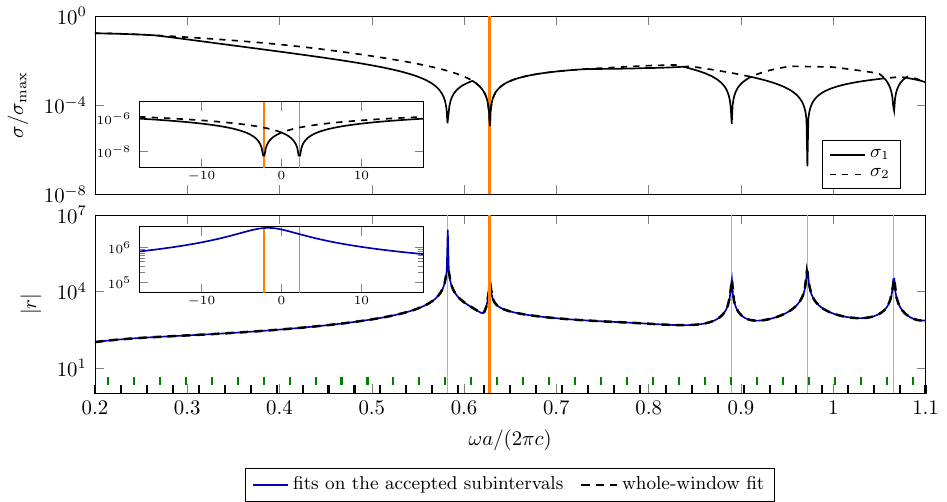}
\caption{The search of \S\ref{sec:sampling} on the rod lattice at $\kk a/\pi=(0.003,0)$.  Top: the two smallest
singular values of the dense matrix $T(\omega,\kk)$.  Bottom: the AAA fits on the whole window and on the
accepted subintervals, with the $33$ initial samples (black) and the $32$ added on the halves (green).
Grey lines mark the other four Bloch frequencies, and orange lines the pair, $4.5\times10^{-6}$ apart.
The insets zoom on the pair, against the offset $10^6(\omega a/(2\pi c)-0.6278043)$.  The fits have a
single pole there, and secant refinement with the local check recovers both frequencies.}

\label{fig:search-anatomy}
\end{figure}

\subsection{A triangular lattice requiring curved walls}
\label{sec:ex-triangular}
Our last example is a triangular lattice, spanned by $\aaa_1=a(1,0)$ and
$\aaa_2=a(\tfrac12,\tfrac{\sqrt3}2)$, of circular air columns of radius $0.48a$ in a dielectric of
permittivity $\varepsilon_1=13$.  Its band gaps in the two polarizations overlap in a complete gap with
gap-to-midgap ratio $18.6\%$~\cite[Ch.~5, pp.~75--76]{JJWM2008}.
Between three neighbouring columns, the dielectric forms spots joined by narrow veins
(Figure~\ref{fig:experiments}).  The isolated spots favour a gap in E-polarization, and the connected
veins a gap in H-polarization~\cite[Ch.~5]{JJWM2008}.  We deform the columns into the pores
\[
 r_\beta(\theta)\,(\cos\theta,\sin\theta),\qquad
 r_\beta(\theta)=0.48a-\beta\,(1-\cos6\theta),\qquad\theta\in[0,2\pi),
\]
centred at the lattice points.  Toward the six nearest neighbours, $\theta=j\pi/3$, the radius stays
$0.48a$, so the narrowest vein is $0.04a$ wide for every $\beta$.  Toward the spots, the boundary moves
inward by $2\beta$, and it becomes non-convex for $\beta>0.48a/38\approx0.0126a$.  We validate against
the crystal of~\cite{JJWM2008}, $\beta=0$, and study $\beta=0.04a$, for which the dielectric fills
$29.5\%$ of the cell instead of $16.4\%$ (Figure~\ref{fig:experiments}).  We compute both polarizations.
Since the pores are air, $\varepsilon_2=1$, H-polarization has $\eta=\varepsilon_1/\varepsilon_2=13$.%

Straight walls are not suitable here since the sides of the parallelogram centred on the pore cut the
circle of $\beta=0$ and pass only $2\times10^{-4}a$ from the pore of $\beta=0.04a$.  We then place the
corner at the midpoint of a vein, $\mathbf c=-(\aaa_1+\aaa_2)/2$, and use the curved walls
\begin{gather*}
 \mathbf r_1(t)=\mathbf c+t\aaa_2+\chi_1(t)\,\mathbf m_2,\qquad
 \mathbf r_2(t)=\mathbf c+t\aaa_1+\chi_2(t)\,\mathbf m_1,\qquad t\in I_1=I_2=[0,1],\\
 \chi_2(t)=-\chi_1(t)=b\,a\sin(2\pi t),
\end{gather*}
in the notation of~\eqref{eq:wall-param}, where $\mathbf m_1=(0,-1)$ and $\mathbf m_2=(\tfrac{\sqrt3}2,-\tfrac12)$ are
the unit vectors to the right of $\aaa_1$ and $\aaa_2$, and $b=-0.0956$ for $\beta=0.04a$ and $b=-0.0929$ for $\beta=0$
maximizes the distance from the walls to the pores.  The walls cross the veins at their midpoints,
where they attain their shortest distance to the pores, $0.02a$.  The Brillouin zone is sampled along $\Gamma$--M--K--$\Gamma$, with
$\mathrm M=(0,\tfrac{2\pi}{\sqrt3a})$ and $\mathrm K=(\tfrac{2\pi}{3a},\tfrac{2\pi}{\sqrt3a})$, with
$10$, $6$ and $12$ intervals on the three segments, that is, $28$ distinct Bloch vectors about
$0.35/a$ apart (Figure~\ref{fig:experiments}, left).

\begin{figure}[tb]
\centering
\raisebox{-0.5\height}{\includegraphics{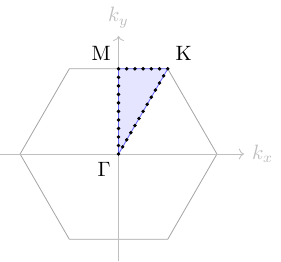}}
\hspace{0.04\textwidth}
\raisebox{-0.5\height}{\includegraphics{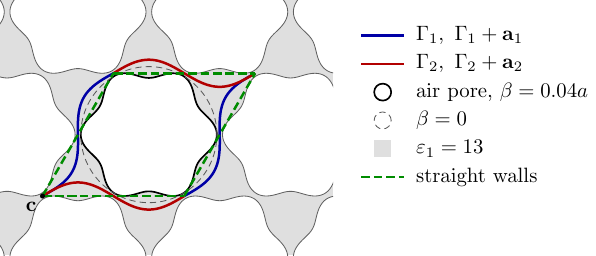}}
\caption{Left: Brillouin zone of the triangular lattice, irreducible zone (shaded) and the $28$ Bloch
vectors of the path $\Gamma$--M--K--$\Gamma$.  Right: the cell of \S\ref{sec:ex-triangular}, one air pore with
$\beta=0.04a$ and its nearest translates, the circle of radius $0.48a$ of $\beta=0$ (dashed), and the
curved walls with their translates; the corner $\mathbf c$ lies at the midpoint of a vein.  Dashed: the
straight walls used in Table~\ref{tab:triangular}.}
\label{fig:experiments}
\end{figure}

This crystal tests the main features of the method, namely a non-convex pore, an oblique
lattice, $\eta\ne1$ in the transmission rows, and walls, corners of $\partial U$ included, within
$0.02a\approx0.4h$ of the pores, where density interpolation (\S\ref{sec:discretization}) evaluates the
nearly singular integrals.  Since the reference value of~\cite{JJWM2008} is a complete gap, it also
tests that no band of either polarization is missed.

Unless stated otherwise, the values $h=0.05a$ and $q=9$ ($N=1040$ for $\beta=0.04a$) are used in the numerical experiments that follow.  We first repeated the empty-lattice test on the oblique lattice.  With $\varepsilon_2=\varepsilon_1=13$ the pore of
$\beta=0.04a$ is invisible, and at M the search returns the three folded-cone frequencies
$|\kk+\mathbf G|/(2\pi\sqrt{13})$ in $[0.01,0.45]$, each degenerate pair once, with errors at most
$4.8\times10^{-10}$ and field residuals below $4\times10^{-9}$.

For the actual crystal, with air pores and $\beta=0.04a$, Table~\ref{tab:triangular} gives the Bloch
frequencies at $\kk a/\pi=(0.3,0.5)$, where the window $[0.01,0.8]$ contains eight E- and five H-polarized frequencies.  Halving $h$ changes them by at most $2.1\times10^{-7}$ (E) and $4.3\times10^{-6}$ (H) and divides the field residuals by about $25$.
Replacing the curved walls by the straight ones, which meet at $\mathbf c$ at $60^\circ$ instead of
$2^\circ$, changes the frequencies by at most $1.1\times10^{-6}$ (E) and $4.2\times10^{-5}$ (H) at $h=0.05a$, and by $1.4\times10^{-8}$ and
$1.5\times10^{-7}$ at $h=0.025a$, below the discretization error of the curved walls at $h=0.05a$.  For
the straight walls in E-polarization at $h=0.025a$, $33$ samples per interval miss one frequency, so
this comparison uses $65$.  The choice of walls thus affects the frequencies only through the discretization error, and curving
the walls away from the pores keeps this error small on a coarse mesh.

Next, Figure~\ref{fig:triangular-bands} shows the bands along $\Gamma$--M--K--$\Gamma$ in $[0.01,0.8]$,
searched with $65$ samples per interval, keeping the frequencies whose three field residuals are at
most $10^{-5}$.  The kept frequencies, $168$ (E) and $93$ (H) for $\beta=0$ and $222$ and $152$ for
$\beta=0.04a$, have residuals at most $2.6\times10^{-6}$ and last secant steps at most $3\times10^{-8}$
($1.7\times10^{-6}$ in H-polarization for $\beta=0.04a$).  For $\beta=0$, refinement also returned two spurious frequencies, about $0.6900$ and $0.7996$, at nearly every Bloch vector and in both polarizations.  They do not depend on $\kk$, so they are not actual
bands.  The matrix $T(\omega,\kk)$ is nearly singular there, but the computed fields do not vanish outside the cell, and
their residuals, between $3.4\times10^{-5}$ and $2.5\times10^{-3}$, exceed the threshold, so the test of
\S\ref{sec:screening} removes them.  They do not occur for $\beta=0.04a$.

With $33$ samples per interval, as in \S\ref{sec:ex-bench}, the search finds the same frequencies for
$\beta=0$, but for $\beta=0.04a$ it misses $19$ E- and $15$ H-polarized frequencies, at $5$ and $4$ of
the $28$ Bloch vectors.  The cause is close pairs.  Between M and K, for instance, the H-polarized
frequencies $0.46496$ and $0.48260$ are closer than the sample spacing $0.025$ and appear as a single
pole, so the bisection stops.  The misses are not due to the accuracy of the samples, since assembling $T(\omega,\kk)$ without
compression and solving by LU instead of GMRES, or halving $h$, misses the same frequencies, while
$65$ samples find them.  To confirm the gaps, we searched each gap window again at all $28$ Bloch vectors with two independent pairs of probes \eqref{eq:scalarized}, and found no frequency inside any gap.

Table~\ref{tab:triangular-gaps} provides the gaps. In both crystals the E-polarized gap, between bands $2$
(maximum at $\Gamma$) and $3$ (minimum at K), lies inside the H-polarized gap, between bands $1$
(maximum at K) and $2$ (minimum at M), as in~\cite[Ch.~5]{JJWM2008}, so the complete gap is the
E-polarized one.  For $\beta=0$ it is $18.96\%$, against $18.6\%$ in~\cite{JJWM2008}.  The difference is
due neither to the sampling of the path, since both edges lie at vertices of the irreducible Brillouin
zone, nor to discretization, since $h=0.025a$ changes the edges by at most $1.1\times10^{-7}$.  For
$\beta=0.04a$ the complete gap shrinks to $0.3589$--$0.4206$, or $15.84\%$.
Finally, Figure~\ref{fig:triangular-modes} shows the E-polarized modes at its edges, evaluated from
\eqref{eq:rep-cell} and \eqref{eq:rep-interior} with density interpolation.  The second band at $\Gamma$
is concentrated in the dielectric spots, with alternating signs, and the third band at K lies mostly in
the pores.  Both are smooth across the walls, with no artifacts near $\mathbf c$ or the vein crossings.

\begin{table}[tb]
\centering\small
\begin{tabular}{llccccc}
  \toprule
  & & \multicolumn{2}{c}{field residual} & & \multicolumn{2}{c}{straight walls} \\
  \cmidrule(lr){3-4}\cmidrule(lr){6-7}
  & $\omega a/(2\pi c)$ & $h=0.05a$ & $h=0.025a$ & $h\to h/2$ & $h=0.05a$ & $h=0.025a$ \\
  \midrule
  \multirow{8}{*}{E} & $0.1334984291$ & $1.7\times10^{-7}$ & $3.0\times10^{-9}$ & $4.2\times10^{-9}$ & $8.8\times10^{-9}$ & $2.2\times10^{-10}$ \\
   & $0.3296942764$ & $4.0\times10^{-8}$ & $2.9\times10^{-9}$ & $5.3\times10^{-8}$ & $1.4\times10^{-7}$ & $1.4\times10^{-9}$ \\
   & $0.4678061655$ & $7.8\times10^{-8}$ & $2.0\times10^{-9}$ & $2.1\times10^{-7}$ & $1.1\times10^{-6}$ & $6.2\times10^{-9}$ \\
   & $0.5004011835$ & $9.5\times10^{-8}$ & $4.6\times10^{-9}$ & $1.5\times10^{-8}$ & $5.3\times10^{-8}$ & $9.5\times10^{-10}$ \\
   & $0.5725763436$ & $2.8\times10^{-8}$ & $1.4\times10^{-9}$ & $4.1\times10^{-8}$ & $1.1\times10^{-7}$ & $2.0\times10^{-9}$ \\
   & $0.6094885117$ & $8.1\times10^{-8}$ & $3.6\times10^{-9}$ & $5.6\times10^{-8}$ & $1.1\times10^{-7}$ & $1.4\times10^{-8}$ \\
   & $0.7468954988$ & $8.7\times10^{-8}$ & $3.4\times10^{-9}$ & $9.3\times10^{-8}$ & $5.6\times10^{-7}$ & $1.0\times10^{-10}$ \\
   & $0.7833974636$ & $1.3\times10^{-7}$ & $6.7\times10^{-9}$ & $4.1\times10^{-8}$ & $7.1\times10^{-9}$ & $2.7\times10^{-9}$ \\
  \midrule
  \multirow{5}{*}{H} & $0.1744884113$ & $2.0\times10^{-7}$ & $3.1\times10^{-8}$ & $5.7\times10^{-7}$ & $6.7\times10^{-6}$ & $1.5\times10^{-7}$ \\
   & $0.5232515045$ & $1.3\times10^{-6}$ & $3.7\times10^{-8}$ & $4.3\times10^{-6}$ & $4.2\times10^{-5}$ & $1.3\times10^{-7}$ \\
   & $0.5465024337$ & $1.0\times10^{-6}$ & $5.7\times10^{-8}$ & $9.0\times10^{-7}$ & $3.9\times10^{-6}$ & $4.5\times10^{-8}$ \\
   & $0.6726586276$ & $1.2\times10^{-7}$ & $5.2\times10^{-9}$ & $1.2\times10^{-7}$ & $2.9\times10^{-6}$ & $1.5\times10^{-9}$ \\
   & $0.6936202922$ & $8.2\times10^{-7}$ & $9.7\times10^{-9}$ & $5.7\times10^{-7}$ & $1.9\times10^{-6}$ & $2.0\times10^{-8}$ \\
  \bottomrule
\end{tabular}
\caption{Bloch frequencies of the crystal of Figure~\ref{fig:experiments} ($\beta=0.04a$) at
$\kk a/\pi=(0.3,0.5)$ in $\omega a/(2\pi c)\in[0.01,0.8]$, computed with $h=0.05a$ and $q=9$ ($N=1040$).
The second column gives the frequencies, the third and fourth the largest of the three field residuals
of \S\ref{sec:screening} at $h=0.05a$ and at $h=0.025a$ ($N=2070$), the fifth the change of the frequency
when $h$ is halved, and the sixth and seventh its change when the curved walls are replaced by the
straight ones, at $h=0.05a$ and at $h=0.025a$.}
\label{tab:triangular}
\end{table}

\begin{figure}[tb]
\centering
\includegraphics{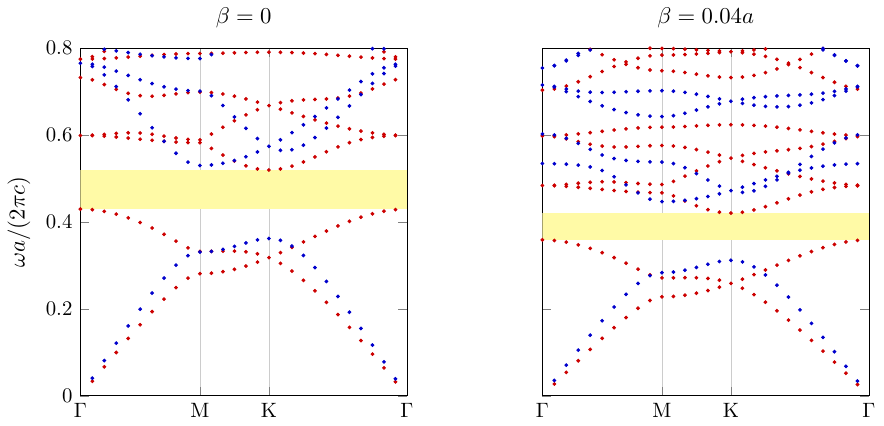}
\caption{Bands of the crystal of Figure~\ref{fig:experiments} along $\Gamma$--M--K--$\Gamma$ for
$\beta=0$ (left) and $\beta=0.04a$ (right), E-polarization (red) and H-polarization (blue), $h=0.05a$,
$q=9$, $65$ samples per interval, frequencies with field residuals at most $10^{-5}$.  The shaded band
is the complete gap.  Degenerate bands appear once, and the acoustic bands at
$\Gamma$ lie below the window.}
\label{fig:triangular-bands}
\end{figure}

\begin{table}[tb]
\centering\small
\begin{tabular}{lcccccc}
  \toprule
  & \multicolumn{2}{c}{E-polarization} & \multicolumn{2}{c}{H-polarization} & \multicolumn{2}{c}{complete gap} \\
  \cmidrule(lr){2-3}\cmidrule(lr){4-5}\cmidrule(lr){6-7}
  & edges & gap & edges & gap & gap & \cite{JJWM2008} \\
  \midrule
  $\beta=0$     & $0.4297$--$0.5197$ & $18.96\%$ & $0.3620$--$0.5300$ & $37.67\%$ & $18.96\%$ & $18.6\%$ \\
  $\beta=0.04a$ & $0.3589$--$0.4206$ & $15.84\%$ & $0.3119$--$0.4469$ & $35.58\%$ & $15.84\%$ & \\
  \bottomrule
\end{tabular}
\caption{Gaps of the crystal of Figure~\ref{fig:experiments}, for $\beta=0$ and $\beta=0.04a$, with their edges in
$\omega a/(2\pi c)$ and their gap-to-midgap ratios.}
\label{tab:triangular-gaps}
\end{table}

\begin{figure}[tb]
\centering
\begin{tabular}{cc}
  \includegraphics{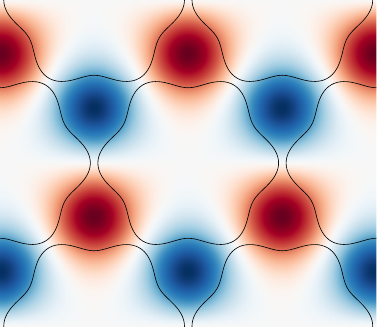} &\qquad
  \includegraphics{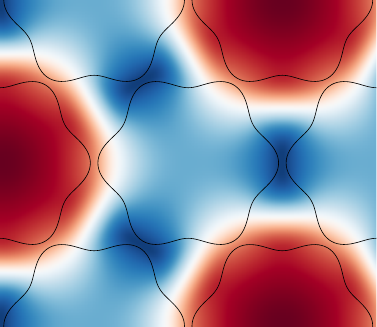} \\
  \small $\Gamma$, $\omega a/(2\pi c)=0.3589$ (lower edge) & \small K, $\omega a/(2\pi c)=0.4206$ (upper edge)
\end{tabular}
\caption{E-polarized Bloch modes at the edges of the complete gap of the crystal of
Figure~\ref{fig:experiments} ($\beta=0.04a$), computed with $h=0.05a$ and $q=9$.  Each panel shows the
real part of the mode on $[-a,a]\times[-\tfrac{\sqrt3}2a,\tfrac{\sqrt3}2a]$, centered on a pore (pore
boundaries in black), after normalizing the mode to maximum modulus $1$ and making it real and positive
where its modulus is largest.  Colors range from dark blue ($-1$) through white ($0$) to dark red ($1$).}
\label{fig:triangular-modes}
\end{figure}

\section{Conclusion}
\label{sec:conclusion}
We have presented a direct boundary integral formulation of the Bloch eigenvalue problem for
two-dimensional photonic crystals that uses only the free-space Green function.  Green's formula on
three-cell stacks yields wall equations free of corner singularities, and with the M\"uller combination
on the inclusion boundary the system is of the second kind.  The Bloch vector enters through finitely
many monomials, so the operators are assembled once per frequency, and spurious null vectors are
detected by an inexpensive residual test.  Combined with a high-order Nystr\"om discretization and an
AAA search, the method reproduces exact and published bands with high accuracy, and handles non-convex pores whose cell
requires curved walls.  For close bands, the reliability of the search is limited by the sample
spacing rather than by the accuracy of the samples.

The cost is dominated by the assembly at every sampled frequency. An efficiency improvement would be to use the $\mathcal H$-tensor format
of~\cite{BHM2026} (see also~\cite{BHM2025}), which would reduce the frequency sweep to little more than
matrix--vector products, and its linearization would give the Bloch frequencies directly as
eigenvalues.  Since only the boundaries are discretized, the method also suits shape uncertainty
quantification of band gaps~\cite{HSSS2018,Sca2019,Sca2022}, for instance with pore contours extracted
from images of manufactured crystals~\cite{CSRSVV2024}.  Finally, the stack construction extends to the
three-dimensional Helmholtz equation, with one stack per lattice direction, and similarly to the
Maxwell equations, for which the M\"uller combination was originally introduced~\cite{Mueller1969}.

\section*{Acknowledgments}
TY gratefully acknowledges support from the National Key R\&D Program of China
(No.~2024YFA1016000), the Beijing Natural Science Foundation (No.~JR26005) and the Strategic Priority
Research Program of the Chinese Academy of Sciences (No.~XDB0640000).

\paragraph{Use of AI tools.}
The authors used the AI assistant Claude (Anthropic) as an aid in writing the code specific to this
paper, which is built on the libraries \texttt{Inti.jl}, \texttt{HMatrices.jl} and
\texttt{Rational\allowbreak Function\allowbreak Approximation.jl}, in producing the figures with TikZ and PGFPlots, and in editing the
text. The authors reviewed all its output and take full responsibility for the content of this paper.

\bibliographystyle{abbrv}
\bibliography{References}
\end{document}